\documentclass[preprint,onefignum,onetabnum]{siamart220329}

\usepackage{lipsum}
\usepackage{amsfonts}
\usepackage{graphicx}
\usepackage{epstopdf}
\usepackage{algorithmic}
\ifpdf
  \DeclareGraphicsExtensions{.eps,.pdf,.png,.jpg}
\else
  \DeclareGraphicsExtensions{.eps}
\fi

\newsiamremark{remark}{Remark}
\newsiamremark{hypothesis}{Hypothesis}
\crefname{hypothesis}{Hypothesis}{Hypotheses}
\newsiamthm{claim}{Claim}

\usepackage{amsopn}

\DeclareMathOperator*{\argmin}{argmin}

\newcommand{\R}{\mathbb{R}}

\newcommand{\N}{\mathbb{N}}

\newcommand{\Fcal}{\mathcal{F}}

\newcommand{\Kcal}{\mathcal{K}}

\newcommand{\Scal}{\mathcal{S}}

\newcommand{\va}{\boldsymbol{a}}

\newcommand{\f}{\boldsymbol{f}}
\newcommand{\g}{\boldsymbol{g}}
\newcommand{\x}{\boldsymbol{x}}
\newcommand{\y}{\boldsymbol{y}}

\newcommand{\vr}{\boldsymbol{r}}

\newcommand{\uu}{\boldsymbol{u}}
\newcommand{\vv}{\boldsymbol{v}}
\newcommand{\vpsi}{\boldsymbol{\Psi}}
\newcommand{\vk}{\boldsymbol{k}}

\newcommand{\ttheta}{\boldsymbol{\theta}}

\newcommand{\pk}{PK-NN}

\usepackage{color}
\usepackage{subcaption}
\usepackage{amsmath}
\usepackage{amssymb}
\usepackage{enumitem}
\usepackage{bbm}
\usepackage{bm}
\usepackage{hyperref}
\usepackage{verbatim}
\usepackage{float}
\usepackage{booktabs}
\usepackage{multirow}
\usepackage{siunitx}
\usepackage{diagbox} 

\usepackage{xcolor}
\usepackage{cleveref}

\usepackage{lineno}
\nolinenumbers

\usepackage{marginnote}

\title{Learning Parametric Koopman Decompositions
for Prediction and Control\thanks{
\funding{This research is part of the programme DesCartes and is supported by the National Research Foundation, Prime Minister's Office, Singapore under its Campus for Research Excellence and Technological Enterprise (CREATE) programme.
GY is supported by the National Research Foundation,
Singapore, under the NRF fellowship (project No. NRF-NRFF13-2021-0005).
The work of IGK is partially supported by the US Department of Energy
and the US Air Force Office of Scientific Research. The work of MK is supported by the AI Interdisciplinary Institute ANITI funding, through the
French “Investing for the Future PIA3” program under the Grant agreement n$^\circ$ ANR-19-PI3A-0004. This work is also co-funded by the European Union under the project ROBOPROX (reg.~no.~CZ.02.01.01/00/22\_008/0004590).}
}}

\headers{Learning Parametric Koopman Decompositions}{Y. Guo, M. Korda, I.G. Kevrekidis, Q. Li.}

\author{
    Yue Guo\thanks{Department of Mathematics, National University of Singapore, 117543, Singapore (\email{guoyue@u.nus.edu}).}
    \and
    Milan Korda\thanks{CNRS, LAAS, 7 avenue du colonel Roche, F-31400 Toulouse, France; Faculty of Electrical Engineering, Czech Technical University in Prague, Prague, Czech Republic; CNRS@CREATE LTD, 1 Create Way, CREATE Tower, 138602, Singapore (\email{korda@laas.fr}).}
    \and
    Ioannis G. Kevrekidis\thanks{Department of Chemical and Biomolecular Engineering, Johns Hopkins University, 3400 North Charles Street
    Baltimore, MD 21218, USA (\email{yannisk@jhu.edu}).}
    \and
    Qianxiao Li\thanks{Department of Mathematics \& Institute for Functional Intelligent Materials, National University of Singapore, 117543, Singapore; CNRS@CREATE LTD, 1 Create Way, CREATE Tower, 138602, Singapore (\email{qianxiao@nus.edu.sg}).}
}
\begin{document}
\maketitle

\begin{abstract}
    We present an approach to constructing approximate Koopman-type decompositions
    for dynamical systems depending on static or time-varying parameters.
    Our method simultaneously constructs
    an invariant subspace \emph{and} a parametric family of projected Koopman operators
    acting on this subspace.
    We parametrize \emph{both the projected Koopman operator family and
    the dictionary that spans the invariant subspace}
    by neural networks, and jointly train them with trajectory data.
    We show theoretically the validity of our approach,
    and demonstrate
    via numerical experiments that
    it exhibits significant improvements over existing methods
    in solving prediction problems,
    especially those with large state or parameter dimensions,
    and those possessing strongly non-linear dynamics.
    Moreover, our method enables data-driven solution of optimal control
    problems involving non-linear dynamics,
    with some interesting implications on controllability.
\end{abstract}

\begin{keywords}
Koopman Operator, Non-autonomous Dynamics, Machine Learning, Invariant Subspace, Control
\end{keywords}

\begin{MSCcodes}
47N70, 37N35, 49M99
\end{MSCcodes}

\section{Introduction}
\label{sec:introduction}

Parametric models play a crucial role in modelling dynamical processes,
allowing one to analyze, optimize, and control them by capturing the relationship
between system behaviour and input parameters.
However, in many scenarios, complete knowledge of the dynamics
may be unavailable and only trajectory data is accessible.
Discovering the relationship
between states and parameters by data-driven approaches is particularly challenging,
especially in non-linear and high-dimensional systems.
The Koopman operator approach~\cite{koopman1931hamiltonian}, initially developed
to convert autonomous non-linear dynamical systems into infinite-dimensional linear
systems has emerged as a powerful tool for spectral analysis and identification
of significant dynamic modes for autonomous systems
\cite{rowley2009spectral, kaiser2017data, budivsic2012applied,
parmar2020survey}.
Various data-driven strategies,
such as dynamic mode decomposition (DMD)~\cite{schmid2010dynamic,
rowley2009spectral, tu2014dynamic} and extended DMD (EDMD)
\cite{williams2014kernel, williams2015data}, have been proposed for
approximating the Koopman operator from data.

In this study, we propose an extension of Koopman operator's application to parametric \emph{
discrete-time} dynamics.
To parametrize and approximate a parameter-dependent Koopman operator
within an invariant subspace, we propose a learning-based method that combines
the ideas in extended dynamic mode decomposition with dictionary learning
(EDMD-DL)~\cite{li2017extended} and a general form of the parametric Koopman
operator that is expanded on a set of \textit{fixed} basis functions
\cite{williams2016extending}.
Our approach is suited for high-dimensional
and strongly non-linear systems, and for both forward non-autonomous prediction
problems and optimal control problems.
Currently, there are many works on the incorporation of parameters in
Koopman decomposition, but they are mostly limited to linear or
bilinear dynamics in the observable space.
Some studies focus on the linear form related to DMD with
controls~\cite{proctor2016dynamic, proctor2018generalizing} and EDMD with
control~\cite{korda2018linear}, which has been applied to system identification
\cite{arbabi2018data, folkestad2020extended, son2021application, lin2021data,
narasingam2020application}, especially for soft robotic system
\cite{haggerty2020modeling,bruder2021koopman}. Recent studies have introduced various parametric adaptations to DMD, utilizing interpolation techniques for enhancing predictions in non-linear dynamics and across parameter spaces. In~\cite{sun2023parametric}, the authors employ Radial Basis Function (RBF) interpolation for mapping parameters to time snapshots, enabling the incorporation of new parameters with either exact DMD or Kernel DMD. A distinct approach is introduced in \cite{andreuzzi2023dynamic}, where DMD is learned for various parameters, and predictions for new parameters are made by modelling the relationship between parameters and their outcomes. 
Another DMD-based approach emphasizes deriving parameter information from the nearest data points, utilizing methods such as Lagrangian interpolation to improve prediction precision \cite{huhn2023parametric}.
Our approach learns a \emph{non-linear} representation of parameters and states within a linear framework, bypassing the need for additional computations like nearest neighbourhood identification and the limitation of the linearity of DMD. This strategy ensures robust continuity and adaptability in the parameter domain, offering effective generalization to new parameter scenarios. When the system
exhibits deviation from linearity, a bilinear form of Koopman dynamics has been explored for adaptation~\cite{surana2016koopman, goswami2017global, peitz2020data,
bruder2021advantages, folkestad2021koopman, strasser2023robust}.
However, these approaches may not be able to address applications
involving strongly non-linear parametric dynamics.
We tackle this problem by embedding the non-linear parameter dependence
directly into the finite-dimensional approximation of the Koopman operator,
which acts over a constructed common invariant subspace \emph{over the entire parameter space}.
Naturally, the simultaneous construction of a parametric Koopman operator
approximation \emph{as well as} a \emph{common} invariant subspace is challenging.
To resolve this, we leverage the function approximation capabilities of
neural networks and jointly train them from data.
We show that our approach can capture intricate non-linear relationships between
high-dimensional states and parameters,
and can be used to solve, in a data-driven way, both parametrically varying prediction
problems and optimal control problems.

The paper is organized as
follows: In~\cref{sec:problem_formulation}, we introduce our definition of
parametric dynamical systems and the parametric Koopman operator. In
\cref{sec:koopman_decomposition}, we present our method for finding a
common invariant subspace and approximating the projected Koopman operator
on this subspace.
In~\cref{sec:results},
we demonstrate, using a variety of numerical examples,
that our approach provides performance improvements
over existing methods.

\section{Koopman operator family for parametric dynamical systems}
\label{sec:problem_formulation}

We begin by introducing the basic formulation of Koopman operator analysis for parametric dynamical systems that we adopt throughout this paper.

\subsection{Parametric dynamical systems}
\label{sec:ParametricDynamicalSystems}

Let $(X, \Scal, m)$, with $X \subseteq \R^{N_{x}}$
be a finite measure space
where $\Scal$ is the Borel $\sigma$-algebra
and $m$ is a measure on $(X, \Scal)$.
Consider a set $U \subseteq \R^{N_{u}}$ of parameters,
and a corresponding parametric family of
transformations $X \to X$
\begin{align}
    \Fcal = \{\f(\cdot, \uu): \uu \in U\}.
\end{align}
This family of transformations can be used to
define parametric discrete-time dynamics
($\x_{n+1} = \f(\x_{n}, \uu)$, $n=0,1,\dots$)
where $\uu$ remains static,
or control systems
($\x_{n+1} = \f(\x_{n},\uu_{n})$, $n=0,1,\dots$)
where $\uu_n$ changes in discrete steps dynamically.
In both cases, at each time instant $n$
the state evolves according to one
member of the family
$
    \{
        \x \mapsto \f (\x) : \f \in \Fcal
    \}
$.
We note that such discrete dynamics can be
also used to model continuous ones
via time-discretization.
In the case of continuous control systems,
we shall only consider those controls that can be
well-approximated by piecewise constant-in-time
functions, e.g. essentially bounded controls.

\subsection{Parametric Koopman operators}
\label{sec:ParmetricKoopmanOperator}

The main idea of the Koopman operator approach
lies in understanding the \emph{evolution of observables over time},
rather than the evolution of individual states themselves.
To this end, let us consider the space of square-integrable
observables
\begin{align}
    L^{2}(X,m) = \{\phi: X \to \R : \|\phi\|_{L^{2}(X,m)} < \infty\}
\end{align}
with the inner product $\langle\phi, \psi\rangle = \int_{X}\phi(\x)\psi(\x)m(\bm d \x)$
and norm
$\|\phi\|_{L^2(X, m)}=|\langle \phi, \phi \rangle|^{\frac{1}{2}}$.
When there is no ambiguity, we will write $L^2$ to mean $L^2(X,m)$.
It is worth noting that a similar definition can also be provided for complex observables, but the analysis and experiments in this study focus on real-valued observables. Therefore, we use $\R$ as the range of observables.
Given any observable $\phi \in L^{2}$,
for each $\uu$,
we have $\phi(\x_{n+1}) = \phi(\f(\x_{n}, \uu))= \phi \circ \f(\x_{n}, \uu)$,
where the composition is in the $\x$ variable.
Hence, the dynamics $\x_{n+1} = \f(\x_{n}, \uu)$ induces a dynamics on the space of the observables
$\phi(\x_{n+1}) = \Kcal(\uu)\phi(\x_{n})$,
where $\Kcal(\uu) : L^{2} \to L^{2}$
is the \emph{parametric Koopman operator},
defined by
\begin{align}\label{eq:pkoopman_ref}
    \Kcal(\uu)\phi(\x) \triangleq \phi \circ \f(\x, \uu).
\end{align}
For each $\uu$, $\Kcal(\uu)$ is linear
since $\Kcal(\uu)(a \phi_{1} +b \phi_{2}) = a \Kcal(\uu)\phi_{1} + b \Kcal(\uu)\phi_{2}$
for $a,b \in \R$ and $\phi_1,\phi_2 \in L^2$.
\Cref{prop:parametric composition operator}
gives conditions
for $\{ \mathcal K(\uu) \}$ to be a well-defined family
of bounded linear operators on $L^2\to L^2$,
which we denote by $\mathcal{B}(L^2)$.
This follows directly from known results in autonomous systems~\cite{singh1993composition},
but we include its proof in~\cref{proof:Continuous Operators} for completeness.
We hereafter assume that these conditions are satisfied.

\begin{proposition}
\label{prop:parametric composition operator}
    Let $\Fcal = \{\f(\cdot, \uu): \uu \in U\}$
    be a family of non-singular transformations on $X$,
    i.e., for every $S \in \Scal$, $m(\f^{-1}(\cdot, \uu)(S)) = 0$ whenever $m(S)=0$.
    Then, $\mathcal K(\uu) \in \mathcal{B}(L^{2})$
    if and only if there exists $b_{\uu}>0$ such that $m(\f^{-1}(\cdot,
    \uu)(S)) \leq b_{\uu} m(S)$ for every $S \in \Scal$.
\end{proposition}

Just like the parametric family of transformations $\{f(\cdot,\uu)\}$,
the Koopman operator family $\{\Kcal(\uu)\}$ can drive both discrete-time parametric dynamics
($\phi_{n+1} = \Kcal(\uu) \phi_{n}$)
or discrete-time control systems
($\phi_{n+1} = \Kcal(\uu_{n}) \phi_{n}$),
both now in the $L^2$ space of observables.

\section{Invariant subspace and finite-dimensional approximation of Koopman operator}
\label{sec:koopman_decomposition}
The central challenge of the Koopman approach is how to efficiently find an
invariant subspace and approximate a projected Koopman operator acting on this
subspace.  To achieve this for \emph{parametric} cases, we generalize the Extended Dynamic
Mode Decomposition with Dictionary Learning (EDMD-DL)
\cite{li2017extended}, but the family of Koopman evolution
operators are approximated by a matrix-valued function of the parameter, acting
on an invariant subspace of observables spanned by \emph{parameter-independent
dictionaries}. In our approach, both the evolution matrix and the dictionary are parameterized by
neural networks and jointly trained.
Then, we can use the trained parametric Koopman operator
on prediction and control problems
defined on the invariant subspace.

\subsection{Numerical methods for autonomous dynamics}
\label{sec:koopman_EDMD_DL}
We first recall the method in~\cite{li2017extended}
to find an invariant subspace and approximate the
Koopman operator on this subspace
for autonomous (non-parametric) dynamical systems
of the form
\begin{align}
    \label{eq:ori_koopman_evol}
    \begin{split}
    \x_{n+1}&=\f(\x_{n}),\\
    \y_{n}&=\g(\x_{n}).
    \end{split}
\end{align}
Here, $\g$ is a length-$N_y$
vector of $L^2$ observable functions $\y$,
whose evolution we are interested in modelling.
The goal is then to find a finite-dimensional
subspace $H \subset L^2$ that contains the
components of $\g$,
and moreover is \emph{invariant}
under the action of the Koopman operator
$\Kcal(\phi) = \phi \circ \f$,
i.e. $\Kcal(H) \subset H$,
at least approximately.
A simple but effective method~\cite{williams2015data}
is to build $H$ as the span of a
set of dictionary functions
$\{\psi_1, \psi_2, \dots, \psi_{N_\psi}\}$ where $\psi_i\in L^{2}(X,m)$,
with $\psi_{i}=g_i$ for $i=1,\dots,N_y$.
Let us write $\vpsi := (\psi_1, \psi_2, \dots, \psi_{N_\psi})^T$
and consider the subspace
$\text{span}(\vpsi)=\{\va^T\vpsi: \va\in\R^{N_\psi}\}$.
For each
$\phi \in \text{span}(\vpsi)$, we can write
$\Kcal \phi=\va^T\Kcal\vpsi=\va^T\vpsi\circ \f$. If we assume that $\Kcal(\text{span}(\vpsi))\subset \text{span}(\vpsi)$, which is equivalent to the existence of $\vk_i \in \R^{N_\psi}$ such that $\mathcal{K}\psi_i=\vk_i^T\vpsi$ for each $i=1,2,\dots,N_\psi$, then the Koopman operator $\mathcal{K}$ can be represented by a finite dimensional matrix $K\in\R^{N_\psi\times N_\psi}$ whose $i^\text{th}$ row is $\vk_i^T$.
In this case, we call $\text{span}(\vpsi)$ a \emph{Koopman invariant subspace}.

In practice, the matrix $K$, and the invariant subspace, can only be found as approximations.
Concretely, one first takes a sufficiently large, fixed dictionary $\vpsi$.
From the dynamical system (\ref{eq:ori_koopman_evol}),
we collect data pairs
$\{\x^{(m)}_{n+1}, \x^{(m)}_{n}\}_{n, m=0}^{N-1, M-1}$,
where $\x^{(m)}_{n+1} = \f(\x^{(m)}_{n})$ and $\x^{(m)}_{n}$
is the state on the $m^\text{th}$ trajectory at time $n$.
Then, an approximation of the Koopman operator
on this subspace is computed via least squares
\begin{align}
    \label{eq:orikmin}
    \hat{K}=\argmin_{K\in \R^{N_\psi\times N_\psi}}\sum_{n, m=0}^{N-1, M-1}||\vpsi(\x^{(m)}_{n+1})-K\vpsi(\x^{(m)}_{n})||^2.
\end{align}
Assuming the data is in a general position, the solution
is guaranteed to be unique when the number of data pairs is at least equal to or
larger than the dimension of the dictionary $\vpsi$.
Otherwise, a regularizer can be incorporated to ensure uniqueness.

Although the solution for (\ref{eq:orikmin}) is straightforward, choosing the
dictionary set $\vpsi$ is a non-trivial task, especially for high-dimensional
dynamical systems~\cite{korda2018convergence, williams2014kernel}.
Consequently, machine learning techniques have been employed to overcome
this limitation by learning an adaptive dictionary from data~\cite{li2017extended, erichson2019physics, lusch2018deep, otto2019linearly, pan2020physics,
takeishi2017learning}.
The simplest approach~\cite{li2017extended} is to parametrize $\vpsi$ using a
neural network with trainable weights $\ttheta_{\psi}$, so that
$\vpsi(\cdot;\ttheta_\psi)$ is \emph{a set of trainable dictionary functions}.
The method iteratively updates $\ttheta_\psi$ and the matrix $K$ by minimizing the
loss function
\begin{align}
\label{eq:dlloss}
L({K},\ttheta_\psi)=\sum_{n, m=0}^{N-1, M-1}||\vpsi(\x_{n+1}^{(m)};\ttheta_{\psi})-{K}\vpsi(\x_{n}^{(m)};\ttheta_{\psi})||^2,
\end{align}
over $K$ and $\ttheta_\psi$.
We note that other losses can be used to promote different
properties of the learned invariant subspace.
For example, one may focus on minimizing
the worst-case error~\cite{haseli2022temporal},
instead of the average error considered above.

\subsection{Numerical methods for parametric dynamics}
\label{sec:NN_koopman}

We now extend the previous algorithm
to the parametric case
\begin{align}
    \label{eq:param_koopman_evol}
    \begin{split}
    \x_{n+1}&=\f(\x_{n},\uu_{n}),\\
    \y_{n}&=\g(\x_{n}).
    \end{split}
\end{align}
This includes the static setting by setting $\uu_n$ to be constant over the $n$-th time interval.
Now, the natural extension of the method for
autonomous dynamics is to
\begin{enumerate}
    \item
    Find a dictionary $\vpsi$
    whose elements are in $L^2$
    (the first $N_y$ of which are $\g$)
    such that $\text{span}(\vpsi)$ is invariant under $\Kcal(\uu)$
    \emph{for all $\uu \in U$}.
    \item
    Construct a $N_\Psi \times N_\Psi$
    matrix-valued function $K(\uu)$
    that approximates $\Kcal(\uu)$
    on $\text{span}(\vpsi)$.
\end{enumerate}
However, the validity of this procedure is not immediately obvious,
since finding a common invariant subspace for all parameters,
even approximately, may be challenging.
We first show in~\cref{prop:finite_param_koopman} below
that this is theoretically possible under appropriate conditions,
and subsequently in~\cref{sec:results} that
it can be achieved in practice.
In the following, $C_b(X)$ denotes the space of
continuous bounded functions on $X$
and $C(U)$ denotes the space of continuous functions
on $U$. We also assume that the state space $X$ is a finite measure space
and consider observables in
$C_b(X) \subset L^2(X,m)$. The proof is found in~\cref{proof:finite_param_koopman}.

\begin{proposition} 
    \label{prop:finite_param_koopman}
    Let $U$ be compact and $X$ be a finite measurable space,
    the observables of interest
    $\g$ satisfy $g_j \in C_b(X)$ 
    and suppose that
    $f(\x,\cdot): U \to X$ is continuous for $m$-a.e. $\x \in X$.
    Then, for any $\varepsilon > 0$,
    there exists a positive integer $N_\psi > 0$,
    a dictionary $\vpsi = \{\psi_1,\dots,\psi_{N_\psi}\}$
    with $\psi_i \in L^2$,
    a set of vectors $\{\va_{j} \in \R^{N_{\psi}}: j=1,\dots,N_y\}$,
    and a continuous function $K : U \to \R^{N_\Psi \times N_\Psi}$
    such that $g_j = \va_j^T \vpsi$ and for all $\uu \in U$,
    we have
    $\|\Kcal(\uu)g_{j} - \va_{j}^{T}K(\uu)\vpsi\|\leq \varepsilon$.
\end{proposition}

\Cref{prop:finite_param_koopman} shows that it is possible to construct a
parameter-independent subspace such that the one-step evolution
of the observables are approximately closed,
and that the projected Koopman operator on this subspace is a continuous function.
In particular, this suggests that we can use neural networks
to approximate both the dictionary elements as
$\Psi(\cdot;\ttheta_{\psi}): X \to \R^{N_{\psi}}$,
and the projected Koopman operator as
$K(\uu ; \ttheta_K): U \to \R^{N_{\psi} \times N_{\psi}}$.
As before, we fix the first $N_y$ entries in $\vpsi$
to be $\g$, while the remaining $N_\psi - N_y$ entries are learned.
The observables $\g$ can be recovered by $\g = B \vpsi$ with
\begin{align}\label{eq:bdef}
        B=\begin{bmatrix} I_{N_y} & O_{N_y\times (N_\psi-N_y)} \end{bmatrix}_{N_{y} \times N_{\psi}},
\end{align}
where $I_{N_y}$ is a $N_y\times N_y$ identity matrix and $O_{N_y\times (N_\psi - N_y)}$ is a $N_y\times (N_\psi - N_y)$ zero matrix.
We train $\vpsi$ and $K$ by collecting data triples
$
\{\x_{n+1}^{(m)}, \x_{n}^{(m)}, \uu_{n}^{(m)}\}_{n,m=0}^{N-1, M-1}
$,
where $\x_{n+1}^{(m)} = \f(\x_{n}^{(m)}, \uu_{n}^{(m)})$, $\x_{n}^{(m)}$ and
$\uu_{n}^{(m)}$ are states and parameters on the $m^\text{th}$ trajectory at
time $n$.
We then minimize the loss function
\begin{align}
    \label{loss_pk_edmd_dl}
      L(\ttheta_{K},\ttheta_\psi)=\sum_{n, m=0}^{N-1, M-1}||\vpsi(\x_{n+1}^{(m)};\ttheta_{\psi})-K(\uu_{n}^{(m)}; \ttheta_{K})\vpsi(\x_{n}^{(m)};\ttheta_{\psi})||^2,
\end{align}
over $\ttheta_K$ and $\ttheta_\psi$.
A graphical representation of our approach
to learn parametric Koopman dynamics,
which we call \pk, is shown
in~\cref{fig:nn_structure} and the training workflow
is summarised in~\cref{alg:PK-EDMD-DL}.
The weights are initialized using the Glorot Uniform initializer~\cite{glorot2010understanding}.
The architecture of $NN_{K}$ consists of a multi-layer, fully connected neural network,
while $NN_{\psi}$ is a residual network to improve optimization performance.
We use the hyperbolic tangent as the activation function for all hidden layers
and a dense layer without activations as the output layers.
The detailed choices of model sizes differ by application and
are included in the corresponding experimental sections.
The Adam optimizer is used for all experiments~\cite{kingma2014adam}.

\begin{figure}[htb!]
	\centering
	\includegraphics[width=0.8\textwidth]{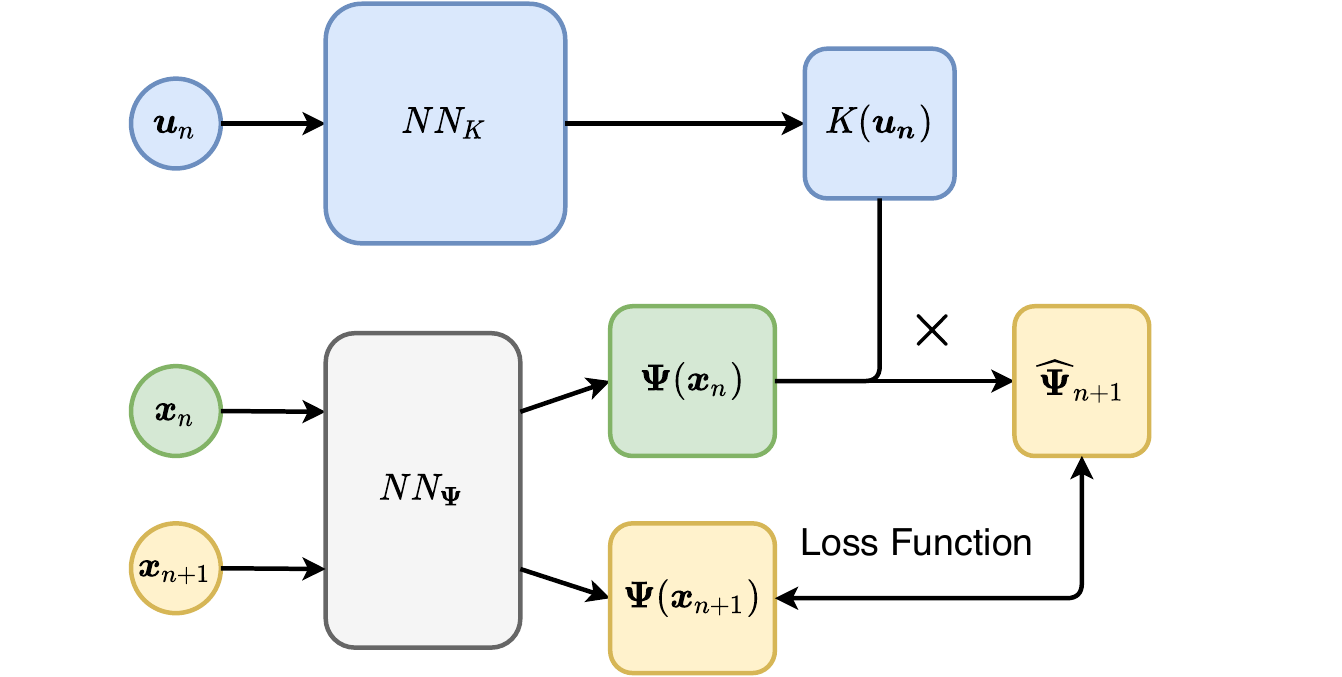}
	\caption{The neural network architecture of \pk.
    The trainable parameters of $K(\uu; \ttheta_{K})$ are integrated within the structure of $NN_{K}$,
    and a distinct neural network $NN_{\vpsi}$ is used to parametrize
    the dictionaries $\vpsi(\x; \ttheta_{\psi})$.}
	\label{fig:nn_structure}
\end{figure}

\begin{algorithm}[H]
\caption{Parametric Koopman decomposition with neural networks (\pk)}
\label{alg:PK-EDMD-DL}
    \begin{algorithmic}
    \STATE{Data: $\{\x_{n+1}^{(m)}, \x_{n}^{(m)}, \uu_{n}^{(m)}\}_{n,m=0}^{N-1, M-1}$}
    \STATE{Initialize: Random $(\ttheta_{K},\ttheta_\psi)$.}
    \STATE{
    Set learning rate $\delta>0$; tolerance $\varepsilon >0$; optimizer \texttt{Opt}.}
    \WHILE{$L(\ttheta_{K},\ttheta_\psi) > \varepsilon$}
    \STATE{Evaluate $L(\ttheta_{K},\ttheta_\psi)=\sum_{n, m=0}^{N-1, M-1}||\vpsi(\x_{n+1}^{(m)};\ttheta_{\psi})-K(\uu_{n}^{(m)}; \ttheta_{K})\vpsi(\x_{n}^{(m)};\ttheta_{\psi})||^2$;}
    \STATE{Update trainable parameters $(\ttheta_{K}, \ttheta_\psi)$ using \texttt{Opt} to minimize $L(\ttheta_{K},\ttheta_\psi)$;
    }
    \ENDWHILE
    \end{algorithmic}
\end{algorithm}

We close the presentation of our approach
with a discussion on related methodologies.
We begin with neural network architectural design.
For the dictionary network $\vpsi$,
we essentially follow the
approaches in~\cite{li2017extended, yeung2019learning}
by incorporating $\g$ into $\vpsi$ via linear transformations,
so that $\g$ can be recovered by a linear operation
(e.g. $\g = B \vpsi$).
When $\g$ is high-dimensional, one may alternatively employ
non-linear variants and recover $\g$ from $\vpsi$
with a trained decoder network~\cite{erichson2019physics, lusch2018deep, otto2019linearly, pan2020physics,
takeishi2017learning}.
These variants can be readily incorporated into our
parametric framework.
For the projected Koopman operator network $K$,
our method can be viewed as an adaptive version
of that introduced in~\cite{williams2016extending},
which writes $K(\uu)$ as a linear combination
$\sum_{i}^{N_{K}} K_i h_i(\uu)$,
with $h_i$ fixed and $K_i$ are fitted from data.
If we use a fully connected network to parametrize $K$,
then our method generalizes this approach
by also training $h_i$.
More generally, our method allows for alternative
architectures for $K$, beyond fixed or adaptive basis expansions.

Let us further contrast our approach with other existing
Koopman operator based approaches for non-autonomous systems.
The first class of methods posits a linear (or Bilinear)
system in the observable space~\cite{proctor2016dynamic, surana2016koopman, korda2018linear,korda2020optimal},
e.g. $\vpsi_{n+1} = A \vpsi_{n} + B \uu_{n}$.
These have the advantage of linearity,
but as we will show in~\cref{sec:results}
that the linearity (or bi-linearity) assumption
maybe too restrictive for some applications. These can be seen as a special case of our approach with $\mathcal{K}(\boldsymbol{u})$ affine in $\boldsymbol u$. A recent generalization of these methods~\cite{cibulka2022dictionary} considers a non-linear transformation of the parameter $\boldsymbol{u}$, i.e., systems of the form $\vpsi_{n+1} = A \vpsi_{n} + B \mathbf{\Phi}(\uu_{n})$.
Another class of methods regard a parametric dynamical
system as an autonomous one in the extended state space
$X \times U$, in the most general case, will
result in parameter-dependent invariant subspaces
(i.e. $\vpsi$ = $\vpsi(\x, \uu)$),
unless some further linearity assumptions are introduced
\cite{proctor2018generalizing}. This framework has been applied to control problems~\cite{shi2022deep}.
% \QL{Milan, can you add something about your previous approaches and how the present one differs here?} \MK{I've put it a bit earlier where it seems to fit better.}
Alternatively, the recent work~\cite{haseli2023modeling}
posits a separable form for the dictionaries
$\vpsi(\x,\uu) = G(\uu)H(\x)$.
If such a form exists, then $H(\x)$ spans
a common invariant subspace.
In this sense, our method (\cref{prop:finite_param_koopman})
gives a different approach to constructing
a common invariant subspace, at least approximately.

\subsection{Parametric Koopman analysis for prediction and control}
\label{sec:application_analysis}
After \pk \ is trained, we can leverage it to solve prediction and control problems in the same family of parametric dynamics.

\paragraph{Prediction problems.}
We start with prediction problems.
In this scenario, an initial value
$\x_{0}$ and one choice of constant parameter $\uu$ or a sequence of parameters
$\{\uu_{n}\}$ are given.
The goal is to predict the sequence of observable values
$\{ \g(\x_{n}) \}$ generated by the underlying dynamics.
Recall that $\g$ is a given vector-valued observable function,
which by design lies in the span of $\vpsi$,
i.e. $\g = B \vpsi$ where $B$ is defined in~\cref{eq:bdef}.
The prediction algorithm is summarized in~\cref{alg:prediction}
for the time-varying parameters case,
and the constant parameter case simply replaces all $\uu_n$
with $\uu$.

\begin{algorithm}[H]
\caption{Prediction by \pk.}
    \label{alg:prediction}
    \begin{algorithmic}
    \STATE{Initial state: $\x_{0}$. Parameters: $\{\uu_{n}\}_{n=0}^{N-1}$.}
    \STATE{$\widehat{\vpsi}_0 =\vpsi\left(\x_{0}\right)$.}
    \FOR{$0 \leq n \leq N-1$}
    \STATE{
        $\widehat{\vpsi}_{n+1}=K(\uu_{n}) \widehat{\vpsi}_{n}$;\\
        $\hat{\y}_{n+1}=B \widehat{\vpsi}_{n+1}$;
    }
    \ENDFOR
    \STATE{Output:$\{\hat{\y}_{n}\}_{n=1}^{N}$.}
    \end{algorithmic}
\end{algorithm}

\paragraph{Optimal control problems.}
In the opposite direction, \pk \ enables us
to solve a variety of inverse-type problems,
where the goal is to perform inference
or optimization on the space of parameters $U$,
e.g. system identification.
In this work, we focus on the particularly challenging class
of such problems in the form of optimal control problems.
Concretely, we consider the discrete Bolza control problem
whose cost function depends on the observable values
\begin{align}
    \label{eq:bolza}
\begin{split}
    \min_{\{\uu_{n}\}_{n=0, 1, \dots, N-1}} \quad &
    J\left[\{\uu_{n}\}\right]
    = \Phi(\g(\x_{N})) + \sum_{n=1}^{N}L_{n} (\g(\x_{n}), \uu_{n-1}) \\
    \text{s.t.} \quad & \x_{n+1} = \f(\x_{n}, \uu_{n}).
\end{split}
\end{align}
Here the initial condition, the terminal cost $\Phi$ and the running costs $L_{n}$ are given,
but the dynamics $\f$ is unknown.
Note that in the fully observed case ($\g(\x)=\x$),
this is a standard Bolza problem.
After the construction of $\vpsi$ and $K$ from data, \pk \ transforms~\cref{eq:bolza} into
\begin{align}
    \label{opt_prob_psi_control}
    \begin{split}
        \min_{\{\uu_{n}\}_{n=0, 1, \dots, N-1}} \quad & J\left[\{\uu_{n}\}\right] =  \Phi(B\widehat{\vpsi}_{N}) + \sum_{n=1}^{N}L_{n} (B\widehat{\vpsi}_{n}, \uu_{n-1}) \\
        \text{s.t.} \quad & \widehat{\vpsi}_{n+1} = K(\uu_{n}) \widehat{\vpsi}_{n},\\
        & \widehat{\vpsi}_{0} = \vpsi(\x_{0}),
    \end{split}
\end{align}
which can be solved using standard optimization libraries.
A concrete example is the tracking problem~\cite{slotine1991applied,korda2018linear},
with the objective of achieving precise following of a desired reference
observable trajectory $\{\vr_{n}\}_{n=0}^{N}$.
In this case, $\Phi\equiv 0$ and $L_n(\y,\uu) = \| \y - \vr_{n} \|^2$.

It is of particular importance to discuss the case where
the control problem does not concern the full state $\x$,
but only some pre-defined observables of it.
A simple example is a tracking problem that only requires
tracking the first coordinate of the state.
In this case, the dictionaries $\vpsi$ are only
required to include the map $\g(\x) = x_1$.
Thus, one can interpret our approach as a
Koopman-operator-assisted model reduction method
for data-driven control.

In the literature,
there exist control approaches which limit to a finite number of control choices,
such as those described in~\cite{peitz2019koopman}
and~\cite{Banks2023}.
These transform the dynamic control problem
into a set of autonomous representations and time-switching optimization problems.
In particular,~\cite{Banks2023} leverages input-parameterized Koopman eigenfunctions to handle systems
with finite control options, embedding the control within the Koopman eigenfunctions to solve the
control problems for the system with fixed points.
In contrast, \pk \ allows for control values to be chosen arbitrarily within a predefined range without being restricted to a set of finite, predetermined options. This feature gives our approach greater flexibility in tackling various control problems.

Compared to existing data-driven approaches,
there are some distinct advantages to transforming
\cref{eq:bolza} into~\cref{opt_prob_psi_control}.
First, compared with methods that learn linear
or bilinear control systems in the observable space~\cite{proctor2016dynamic, surana2016koopman},
the transformed problem~\cref{opt_prob_psi_control}
is valid for general dynamics, where
$\f$ may be strongly non-linear in both the state
and the control.
In these cases, we observe more significant improvements
using our method (See~\cref{sec:control}).
Second, we contrast with the alternative approach of
first learning $\f$ in~\cref{eq:bolza}
using existing non-linear system identification techniques~\cite{chiuso2019system, ljung2020deep, tanyu2022deep},
and then applying optimal control algorithms directly to~\cref{eq:bolza}.
While this approach is feasible for low-dimensional systems,
it becomes inefficient when dealing with high-dimensional
non-linear systems, where the control objectives only
depend on a few low-dimensional observables.
Our method can avoid the expensive
identification of high-dimensional dynamics and
the subsequent even more expensive optimal control problem
in high dimensions.
Instead, we parsimoniously construct a low-dimensional control system
in the form of~\cref{opt_prob_psi_control},
on which the reduced control problem can be easily solved
(See tracking the mass and momentum of Korteweg-De Vries (KdV) equation in~\cref{sec:control}).

\subsection{\pk \ and non-linear controllability}
\label{sec:controllability}

We conclude with a discussion on the interesting subject of controllability.
For control problems, and in particular tracking problems,
it is desirable that the constructed data-driven
state dynamics should be controllable, i.e. the initial state
can be driven to an arbitrary target state, at least locally.
This gives the system flexibility to track a given trajectory.
The
controllability rank condition in the lifting space of the Koopman operator
analysis with a linear dependence on the control is addressed in~\cite{brunton2021modern}.
Nevertheless, such linear models (as we show in~\cref{sec:results})
have limited modelling capabilities in the general non-linear setting,
prompting us to consider the
controllability of \pk \ type Koopman dynamics.
In particular, we demonstrate that the dynamics~\cref{opt_prob_psi_control}
has the special property
that controllability conditions are readily checked,
and may even be algorithmically enforced.

We discuss this in the context where the dynamics
$\x_{n+1} = \f(\x_{n}, \uu_{n})$ is a discretization
of a differential equation
$\dot{\x} = \tilde{\f}(\x,\uu)$ with initial condition $\x(0) = \x_0$. We assume $\tilde{\f}(\x,\uu)$ is continuous in $\uu$ and Lipschitz continuous in $\x$ across the entire domain. This assumption guarantees the existence of a unique solution of $\dot{\x} = \tilde{\f}(\x,\uu)$ by the Picard--Lindelöf theorem and inherently ensures forward completeness \cite{angeli1999forward, gonzalez2021anti}.
Thus, the discrete system $\x_{n+1} = \f(\x_{n}, \uu_{n})$ effectively approximates the continuous system by defining $\f(\x_{n}, \uu_{n})$ as the solution to $\dot{\x} = \tilde{\f}(\x,\uu)$ with $\x(0) = \x_n$ and with $\uu(t) = \uu_n$ being constant over a given sampling interval $\Delta t$.
If the continuous system is controllable, the controllability of the discrete system can be approximately guaranteed within any arbitrary error. 
A classical sufficient condition for controllability is a
corollary of the Chow-Rashevsky theorem~\cite{chow1940systeme, rashevsky1938connecting}:
if the family $\Fcal = \{\tilde{\f}(\cdot, \uu): \uu \in U\}$ is smooth and symmetric
(meaning that $-\tilde{\f} \in \Fcal$ if $\tilde{\f} \in \Fcal$),
then the system is controllable with piecewise constant control
if $\text{dim Lie}_{\x}\Fcal = d$ for each $\x$, where $\text{Lie}_{\x} \Fcal$ is the span of the Lie algebra
generated by $\Fcal$ at $\x$. While useful, this condition is difficult to check
for a given $\tilde{\f}$, as it requires the generation of
Lie algebras at \emph{every} point in the state space,
unless each member of $\Fcal$ is linear.
This is the fundamental difficulty towards ensure, or even verifying
the controllability of non-linear systems.

However, it turns out that controllability is much easier to
check for the transformed system under parametric Koopman operator analysis. Given observable $\phi \in C_b$, for each $\uu$, the Koopman operator is well-defined and can be considered as strongly continuous semigroups $(\mathcal{S}^{t}(\uu))_{t \geq 0}$ with $\mathcal{S}^{t}(\uu) \phi(\x_{0}) = \phi (\x(t))$, $\mathcal{S}^{0}(\uu)\phi(\x) = \phi(\x)$ \cite{hille1996functional, mauroy2020koopman}. There exists a dense subset $D$ (such as $C^{\infty}$) of $C_b(X)$ such that $\dot{\phi} = \lim_{t \to 0^{+}} \frac{\mathcal{S}^{t}(\uu) \phi - \phi}{t} \triangleq \tilde{\Kcal}(\uu)\phi$ exists in the strong sense for $\phi \in D$. This defines the Koopman generator $\tilde{\Kcal}(\uu): D \to C_b(X)$, which is the time derivative over the evolution of the observables. We assume that $\phi$ is approximately in a finite-dimensional subspace, and the components in the dictionary $\vpsi$ span this subspace.
The continuous analogue to \eqref{opt_prob_psi_control} is thus
$\dot{\vpsi} = \tilde{K}(\uu)\vpsi$ where $\tilde{K}(\uu)$ is the projected Koopman generator onto the dictionary $\vpsi$. 
Concretely, we may compute the corresponding projected
Koopman generator
for the continuous system as $\tilde{K}(\uu) = \frac{K(\uu)-I}{\Delta t}$ where $I$ is
an identity matrix.
Furthermore, we can extend this to a finite number of observables $\phi_i $'s which can be represented in this subspace. Controllability in the continuous system $\dot{\vpsi} = \tilde{K}(\uu)\vpsi$ ensures these observables $\phi_i $'s are also controllable.
Then, the control family $\Fcal = \{\tilde{K}(\uu)(\cdot): \uu \in U\}$
is now a family of \emph{linear} functions.
Therefore, the controllability conditions can be checked at one (any)
state, and automatically carries over to all states.
We note that this notion of controllability is not over
the original state space,
but over the (approximately) invariant subspace spanned
by $\vpsi$.
In this sense, we can understand the parametric Koopman approach
as one that
constructs an invariant subspace of observables
with the property that \emph{controllability in this subspace at any point carries over to
the whole subspace}
-- a desirable property of linear control systems.

One way to verify controllability is as follows.
Let $\text{Lie}_{\x}^{(0)} \Fcal \triangleq \text{span}\{\tilde{K}(\uu)\x: \uu \in U\}$.
We know that $\text{Lie}_{\x}^{(0)} \Fcal \subseteq \text{Lie}_{\x}\Fcal$.
Then, $\text{dim Lie}_{\x}\Fcal = d$ can be guaranteed by $\text{dim
Lie}_{\x}^{(0)} \Fcal=d$, which is equivalent to:
for any non-zero vectors
$\vv_{1}, \vv_{2} \in \R^{d}$, there exists $\uu_{j} \in U$ such that
$\vv_{1}^{T} \tilde{K}(\uu_{j}) \vv_{2} \neq 0$.
This can be ensured if the flattened matrices
$\{\tilde K(\uu) : \uu \in U\}$ do not lie in a proper subspace.
Algorithmically, we collect a large sample of parameters
$\uu_i$, $i=1,2,\dots,N$, flatten each matrix
$\tilde{K}(\uu_{i}) \in \R^{d \times d}$ into a vector $\vk_{i} \in \R^{d^2}$
and concatenate all $\vk_{i}^{'}$s as a $d^2\times N$ matrix
\begin{align}
    \label{eq:matrix_C}
    C = [\vk_{1}, \vk_{2}
    \dots, \vk_{N}].
\end{align}
The system is controllable if the rank of $C = d^{2}$.
We show in~\cref{sec:control}
that this can be readily checked for a learned \pk \ system,
and moreover, that it may guide architecture designs for $\uu \mapsto K(\uu)$
that promote controllability,
yielding an approach that
finds a parsimonious controllable dynamics in the observable
space.
It should be noted that although this rank condition is sufficient to guarantee the controllability of the observables, it is not necessary when the observables $\vpsi$ form a manifold. 
Investigating the model structure design and the controllability of the manifold presents an interesting future research direction.

\section{Numerical experiments}
\label{sec:results}

In this section, we present numerical results on a variety of prediction and control problems.
We evaluate various baseline methods,
noting that some of these overlap in different aspects of the comparison.
For convenience, we introduce a standard format to refer to these methods in
\cref{tab:names_of_models}.
The code to reproduce these experiments are found at~\cite{guo2023learningcode}.

\begin{table}[htb!]
    \centering
    \caption{Summary of comparisons. ``NN'' stand for (residual) neural networks
    and ``RBF'' refers to dictionaries built from random radial basis functions. $u_{n,i}$ is the $i^{\text{th}}$ component of $\uu_n$.}
    \label{tab:names_of_models}
    \begin{tabular}{l l c l}
      \toprule
      \toprule
       Method Description & $\vpsi$& Reference & Notation \\
      \midrule
      $\x_{n+1} = K \x_{n}$ (Optimized DMD)
              & $\x$
              &\cite{heas2022low}
              & M0 \\
      \midrule
      \multirow{2}{*}{$\vpsi(\x_{n+1}) = K \vpsi(\x_{n})$(EDMD)}
              & RBF
              &\cite{williams2015data}
              & M1-RBF \\
        & NN
        &\cite{li2017extended}
        & M1-NN \\
        \midrule
       $\vpsi(\x_{n+1})= A \vpsi(\x_{n}) + B \uu_{n}$
              & NN
              &\cite{proctor2016dynamic}
              & M2 \\
        \midrule
       $\vpsi(\x_{n+1})= A \vpsi(\x_{n}) + \sum_{i} B_{i} u_{n,i}\vpsi(\x_{n})$
              & NN
              &\cite{surana2016koopman}
              & M3 \\
              \midrule
      \multirow{2}{*}{}
      $\vpsi(\x_{n+1}) = \sum_{i=1}^{M} h_{i}(\uu_{n})K_{i} \vpsi(\x_{n})$ & RBF & \cite{williams2016extending} & M4-RBF\\
      $M$ is fixed and $h_{i}$ are fixed polynomials  & NN &  & M4-NN \\
      \midrule
       $\vpsi(\x_{n+1}) = K(\uu_n)\vpsi(\x_{n})$ & NN  &  & Ours \\
      \bottomrule
      \bottomrule
    \end{tabular}
  \end{table}

We first demonstrate on prediction problems that \pk \ outperforms
existing methods, including optimized DMD (M0), classical EDMD (M1-RBF) and EDMD with dictionary
learning (M1-NN) in~\cref{sec:K-EDMD-DL}, linear (M2) and bi-linear (M3) extensions
to systems with parameters or control in~\cref{sec:nonlinearity} and a modified
form of EDMD (M4) in~\cref{sec:high_dim}.
In the reverse direction, we show that \pk \
can solve data-driven optimal control problems in~\cref{sec:control}.
Notably, our approach exhibits more significant improvements
for problems with strong non-linearity or those involving high-dimensional states and parameters.
Although the structure of the dictionary is detailed in the following paragraphs, here we highlight the dictionary is designed as $(1, \x, NN)^{T}$ by default.
This consists of a constant element, the state $\x$ and trainable components $NN$ generated by neural networks, the same as the dictionary in~\cite{li2017extended}. Furthermore, we construct $K(\uu)$ by assigning its first row to $(1, 0, \dots, 0)$, which constitutes one of the optimal choices for the first row in $K(\uu)$. We also provide a detailed discussion about selecting hyperparameters in \pk \ for different tasks in~\cref{sec:hyperparam_selection}.

\subsection{Prediction problems}
\label{sec:predction}
Using~\cref{alg:prediction},
we test the performance of \pk \ on
various forward prediction problems
involving discretized parametric ordinary and partial
differential equations.
Here, the discrete-time step $n$ corresponds to the physical time
$t_n$ of the continuous dynamics, and are equally spaced.
We evaluate performance by
the relative reconstruction error at $t_{n}$
on the $m^\text{th}$ trajectory is defined as
\begin{align}
\label{eq:rel_error}
E^{(m)}\left(t_n\right)=\frac{\sqrt{\sum_{i=1}^n\left\|\hat{\y}_i^{(m)}-\y_{i}^{(m)}\right\|^{2}}}{\sqrt{\sum_{i=1}^n\left\|\y_{i}^{(m)}\right\|^{2}}},
\end{align}
where $\y_{i}^{(m)}$ is the true value of observables and $\hat{\y}_{i}^{(m)}$
is the predicted observables at $t_{i}$ on the $m^\text{th}$ trajectory. To
evaluate the performance on $M$ trajectories, we use the average $E(t_{n}) =
\frac{1}{M}\sum_{m=1}^{M}E^{(m)}\left(t_n\right)$. 
If there is no special
emphasis on the observation function, we set $\g(\x) = \x$, the full-state observable.
With the state as the output at each step, we substitute the predicted state into the dictionary and multiply the dictionary by $K(\uu)$ to improve the stability of long-term predictions.

\subsubsection{Improved accuracy and generalization over parameter independent Koopman dynamics}
\label{sec:K-EDMD-DL}

We first demonstrate that \pk \ can interpolate successfully
in parameter space and improve upon naive training
a separate (parameter-independent) Koopman dynamics for each observed $\uu$ using EDMD with dictionary learning (M1-NN)~\cite{li2017extended}, EDMD with radial basis functions (M1-RBF)~\cite{williams2015data} and optimized DMD (M0)~\cite{heas2022low}.
We consider the parametric Duffing equation
\begin{align}
    \label{eq:duffing_param}
    \begin{split}
    \dot{x_1}&=x_2, \\
    \dot{x_2}&=-\delta x_2-x_1\left(\beta+\alpha x_1^2\right).
    \end{split}
\end{align}
To quantify performance,
we use $\nu_{1}$ to denote the number of trajectories
for each set of fixed parameters $\uu = (\delta, \alpha, \beta)$
and $\nu_{2}$ to denote the number of different
parameter configurations in the dataset.
The baseline method trains $\nu_2$ separate Koopman dynamics and chooses one such dynamic with parameter value being the nearest neighbour of the testing parameter for testing.
On the other hand, \pk \ trains only one parametric Koopman dynamics
that can interpolate over the space of the parameters $\uu$
and directly uses~\cref{alg:prediction} to predict.
For each configuration of $\nu_{1}$ and $\nu_{2}$, training data is generated with trajectories initialized within the domain $[-2,2]\times[-2,2]$. The parameters $\delta$, $\alpha$ and $\beta$ are randomly sampled from uniform distributions over the intervals $[0,1]$, $[0,2]$ and $[-2,2]$ respectively.
To ensure
fairness, the total amount of training data $\nu_1\times\nu_2 = 10000$
remains constant.
For each trajectory, we observe the data
points over 50 steps with $\Delta t=0.25$, the same setting as the experiments
in~\cite{li2017extended}.
The dictionary $NN_{\psi}$ is constructed as $(1, x_1, x_2, NN^T)^T$ where $NN$ is a trainable 22-dimensional vector. This $NN$ is the output of a 3-layer feed-forward ResNet with a width of 100 nodes in each layer. The dictionary by this design is employed in both \pk \ and M1-NN. For $NN_{K}$ in \pk, we use a 3-layer fully connected neural network with width 256 before the output layer.

Our results, as shown in~\cref{fig:duffing_param}, indicate that
our method increasingly outperforms the other three approaches as $\nu_2$ increases.
This is because, while we have few representative trajectories
per parameter instance (causing M1-NN and M1-RBF to fail),
the interpolation in parameter space
allows us to integrate information across
trajectories with different parameter values,
thus ensuring prediction fidelity and generalizability.

\begin{figure}[htb!]
	\centering
	\includegraphics[width=1.0\textwidth]{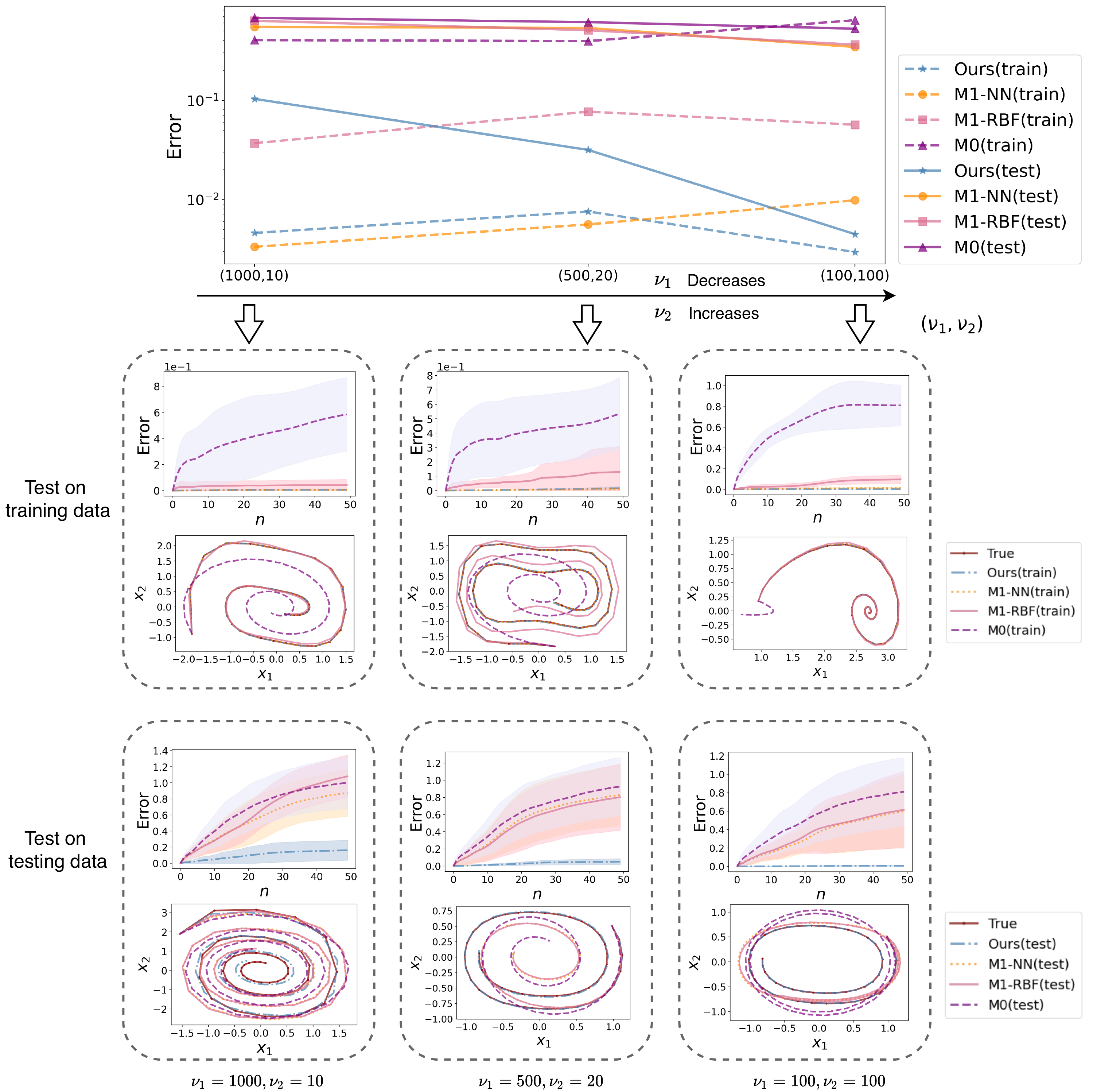}
	\caption{Comparison on different $(\nu_{1}, \nu_{2})$ with $\nu_{1}\nu_{2} =
	10000$. The topmost panel displays the prediction results for \pk \ and M1 on
	both training and testing data. The horizontal axis represents the values of
	$(\nu_{1}, \nu_{2})$.
    In the bottom panels, we provide a detailed error
	analysis at varying time steps across various $(\nu_{1}, \nu_{2})$
	configurations.
    We observe that
    \pk \ out-performs M1-NN, M1-RBF and M0 significantly when $\nu_{2}$ increases,
    which results in fewer trajectories per parameter configuration.}
	\label{fig:duffing_param}
\end{figure}

\subsubsection{Improved performance on strongly non-linear problems}
\label{sec:nonlinearity}

To see the performance of \pk \ on the problems with strong non-linearity,
we test our method on the Van der Pol Mathieu equation~\cite{jianliang2021nonlinear}
\begin{align}
    \label{eq:vdpm}
    \begin{split}
        \dot{x_1}&=x_2,\\
        \dot{x_2}&=(k_{1} - k_{2} x_{1}^{2})x_{2} - (w_{0}^{2} + 2\mu u^{2}- \mu)x_{1} + u k_{3},
    \end{split}
\end{align}
where $k_{1}=2, k_{2}=2, k_{3}=1$ and $w_{0}=1$. The parameter $u$ is designed to adjust both the parametric excitation $(2 \mu u^{2})x_{1}$ and the external excitation $u k_{3}$.
The value of the parameter $\mu >0$ captures the non-linearity among $\x=(x_{1},
x_{2})$ and $u$,
and this non-linearity becomes more pronounced as the parameter
$\mu$ increases.
Training data consists of 500 trajectories with $\Delta t = 0.01$. Each
trajectory spans 50 sampling time steps initialized from a uniformly randomly sampled point in $[-1,1]\times[-1,1]$.
The parameter at each time
step are also randomly sampled from uniform distributions over $[-1,1]$.
In this section, we consider the case where
$\nu_1=1$ as in~\cref{sec:K-EDMD-DL},
thus we do not compare with EDMD using RBF (M1-RBF) or NN (M1-NN), but instead
with the linear Koopman with control approach (M2)~\cite{proctor2016dynamic}
\begin{align}
    \label{koopman_linear}
    \widehat{\vpsi}_{n+1}= A \widehat{\vpsi}_n + B \uu_{n},
\end{align}
and the bi-linear variant (M3)~\cite{surana2016koopman}
\begin{align}
    \label{koopman_bilinear}
    \widehat{\vpsi}_{n+1}= A \widehat{\vpsi}_n + \sum_{i}^{N_{u}} B_{i} u_{n,i} \widehat{\vpsi}_n,
\end{align}
where $u_{n,i}$'s are the components of vector $\uu_{n} \in \R^{N_{u}}$
and $A, B_{i}$'s $\in \R^{N_{\psi}\times N_{\psi}}$.
In experiments, to ensure the effect of dictionaries on state $\x$ is the same,
we use trainable dictionaries, which are built with the same network structure and
the same dictionary number.
The dictionary $NN_{\psi}$ is $(1, x_1, x_2, NN^T)^T$ and the dimension of the trainable vector $NN$ is 10. $NN$ is designed by two hidden layers with width 64 and $NN_{K}$ employs one hidden layer with width 128.
In the experiments,
we randomly initialize the dictionary and use the least squares method to get
the initial $A$, $B$ or $b_{i}$.
To obtain the final results for linear and bi-linear models,
we train the dictionaries and optimize $A$, $B$ or $b_{i}$ iteratively.
For \pk, trainable weights are not only contained in dictionaries but also
in the projected Koopman operator. We train them together to find an approximately optimal model.
As shown in~\cref{fig:vdpm_comparison}, we show that \pk \ does not perform as
well as linear and bi-linear model when $\mu=0$ but improves upon them when
$\mu=1, 2, 3, 4$.
This is expected since the system is closest to a non-linear/bi-linear system when $\mu=0$,
whereas it departs more significantly from a linear structure as $\mu$ increases. 
In~\cref{sec:vdpm_mu0}, we explain that \pk \ can still perform well when $\mu=0$ using a perturbed initialization from a set of precise weights.

\begin{figure}[htb!]
	\centering
	\includegraphics[width=1.0\textwidth]{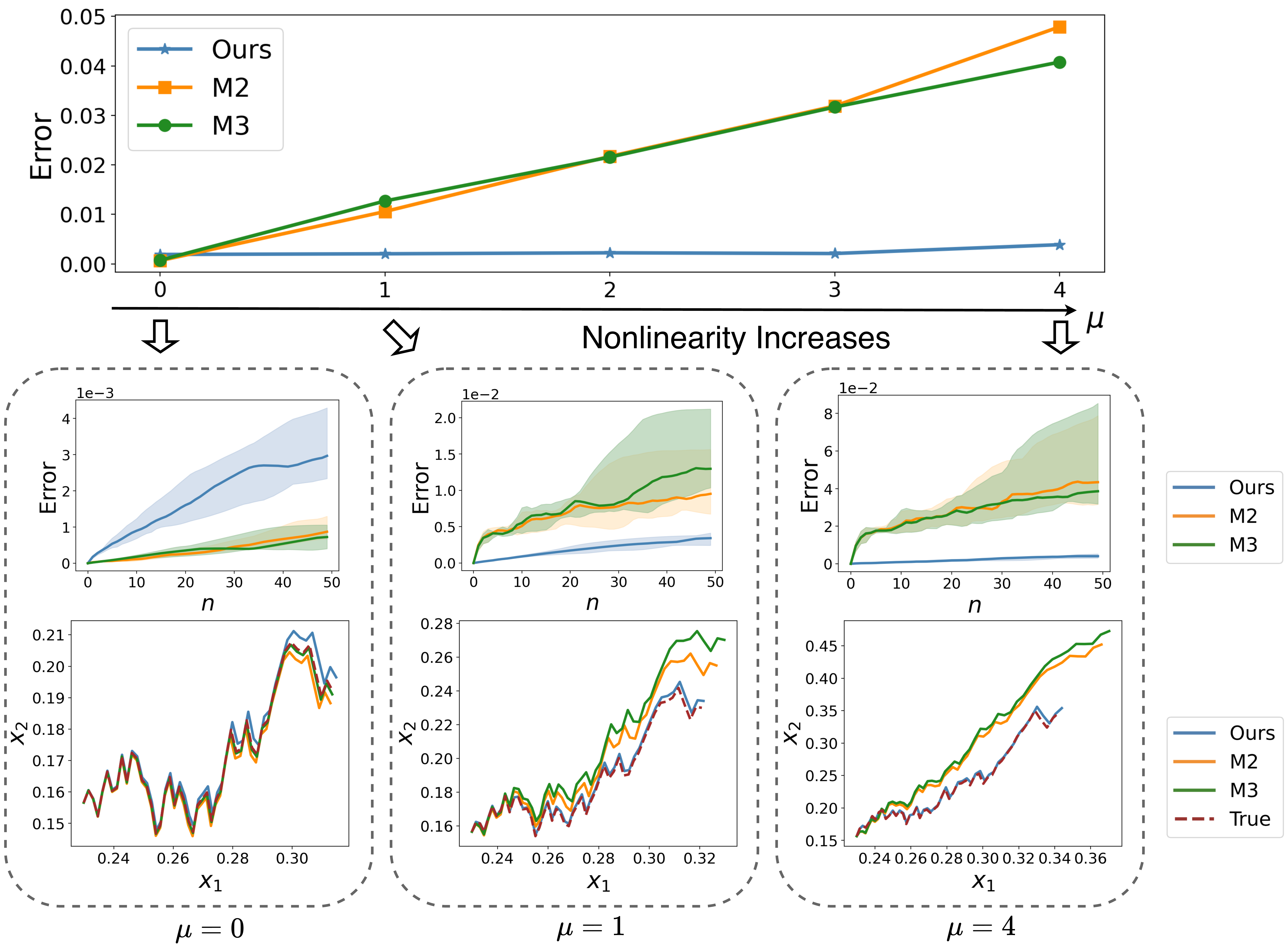}
	\caption{A comparison of errors that occur as the non-linearity of variables
	$\uu$ and $\x$ in the problem increases, with the variable $\mu$ acting the non-linearity coefficient.
    The topmost figure showcases the error analysis for \pk, linear model, and bi-linear model as the non-linearity
	coefficient $\mu$ is incremented. The bottom blocks illustrate a contrast
	between predicted trajectories of different methods and the exact solution,
	along with an error analysis at each time step for the corresponding $\mu$
	configuration.
    \pk \ outperforms the other two structures when the non-linearity of the system increases.}
	\label{fig:vdpm_comparison}
\end{figure}

\subsubsection{Improved accuracy for high-dimensional systems}
\label{sec:high_dim}
We now apply \pk \ to the high-dimensional
problem involving a variation of the FitzHugh-Nagumo partial differential equation subject to quasi-periodic forcing, which is mentioned in~\cite{williams2016extending}
\begin{align}
    \label{eq:fhn_pde}
    \begin{split}
    \partial_t v & =\partial_{x x} v+v-v^3-w+u(t) ( e^{-(x-k_{1})^2 / 2} + e^{-(x-k_{2})^2 / 2} + e^{-(x-k_{3})^2 / 2}), \\
    \partial_t w & =\delta \partial_{x x} w+\epsilon\left(v-a_1 w-a_0\right).
    \end{split}
\end{align}
where $\delta=4, \epsilon=0.03, a_1=2, a_0=-0.03, k_{1}=-5, k_{2}=0, k_{3}=5$ on
the domain $x \in (-10,10)$ with Neumann boundary conditions.
We adopt a finite difference method to discretize the spatial derivatives
on a regular mesh consisting of 10 points.
In this study, the state vector comprises two components: a 10-dimensional discretized activator complex $\boldsymbol{v}$ and a 10-dimensional discretized inhibitor complex $\boldsymbol{w}$, resulting in a 20-dimensional state vector. The initial condition for the activator complex $\boldsymbol{v}$ is given by $\sin(\frac{a \pi x}{10}+\frac{\pi}{2})$ where $a$ is an integer randomly sampled from the range $(1,20)$. The inhibitor complex $\boldsymbol{w}$ is initialized to zero at all spatial points. The parameters $u \in \R$ along the trajectory are uniformly randomly generated in $(-1,1)$.
% \QL{What is the state space? what is the parameter space?}
Training data are generated on 100 trajectories over 500 sampling time-steps and
$\Delta t = 0.001$.
We compare \pk \ with the method introduced in~\cite{williams2016extending},
which uses fixed radial basis functions as the dictionaries on states and
approximates the parametric Koopman operator as
\begin{align}
\label{eq:polyK}
    K(\uu) = \sum_{i=1}^{N_{K}} h_{i}(\uu)K_{i}.
\end{align}
In this experiment, the parameter-dependent functions $h_{i}(\uu)$ are
approximated by polynomials up to third order.
We refer to this method by ``M4-RBF'' in the following figures.
Furthermore, to observe the influence of the dictionaries on states, we
introduce a modification called ``M4-NN'', where
the dictionaries are trainable,
but the structure of $K(\uu)$ is the same as~\cref{eq:polyK}.
In this experiment, the dictionary is designed as $(1, \boldsymbol{v}^T, \boldsymbol{w}^T, NN^T)^T$ where $\boldsymbol{v},\boldsymbol{w} \in \R^{10}$ and $NN$ is 10-dimensional trainable vector. $NN$ utilizes two hidden layers with a width of 128, whereas $NN_{K}$ employs one hidden layer with a width of 16.

In~\cref{fig:fhn_error}(a), we show that trainable dictionaries
lead to better prediction results than only using RBF,
suggesting that adaptive dictionaries are effective
for high-dimensional state spaces.
Additionally, the performance of \pk \ aligns similarly to M4-NN. The reason may be the parameter $\uu$ in Equation~\cref{eq:fhn_pde} is one-dimensional, making the polynomial-based approximation of $K(\uu)$ sufficiently effective. To verify this hypothesis, we modify~\cref{eq:fhn_pde} to
\begin{align}
    \label{eq:fhn_pde_high_u}
    \begin{split}
    \partial_t v & =\partial_{x x} v+v-v^3-w+u_{1}(t) e^{-(x-k_{1})^2 / 2} + u_{2}(t)e^{-(x-k_{2})^2 / 2} + u_{3}(t) e^{-(x-k_{3})^2 / 2}, \\
    \partial_t w & =\delta \partial_{x x} w+\epsilon\left(v-a_1 w-a_0\right).
    \end{split}
\end{align}
with three parameters $\uu = (u_{1}, u_{2}, u_{3}) \in \R^{3}$,
while other settings remain unchanged.
The number of trainable weights in $K$
are kept approximately the same in both methods,
whereas the dimensions of the state dictionary $\vpsi$ are identical. The dictionary retains a structure of $(1, \boldsymbol{v}^T, \boldsymbol{w}^T, NN^T)^T$. $NN$ utilizes two hidden layers with a width of 128 and $NN_{K}$ has two hidden layers with width 16.
In~\cref{fig:fhn_error}(b),
we observe that \pk \ now performs better 
than the two baselines, consistent with our hypothesis.
Expanding our investigation to encompass systems with higher dimensionality,~\cref{fig:fhn_error}(c) shows an improvement in the performance of \pk \ as the spatial discretization range is broadened from 10 to 100 for~\cref{eq:fhn_pde} and \pk \ also outperforms other two methods. The results of the prediction on the original space, with $N_x = 10$ and dim $\uu = 1$, are presented in~\cref{fig:fhn_dim_10_traj_diff}. We can see that the predictions by \pk \ are more accurate than the other methods, M4-RBF and M4-NN. The results on $N_x =10$, dim $\uu=3$ and $N_x =100$, dim $\uu=1$ are shown in~\cref{sec:fhn_traj}. 
% This underscores the proficiency of our approach in handling high-dimensional systems effectively.
This demonstrates the effectiveness of our approach in handling high-dimensional systems.

\begin{figure}[htb!]
    \centering
    \includegraphics[width=1.0\textwidth]{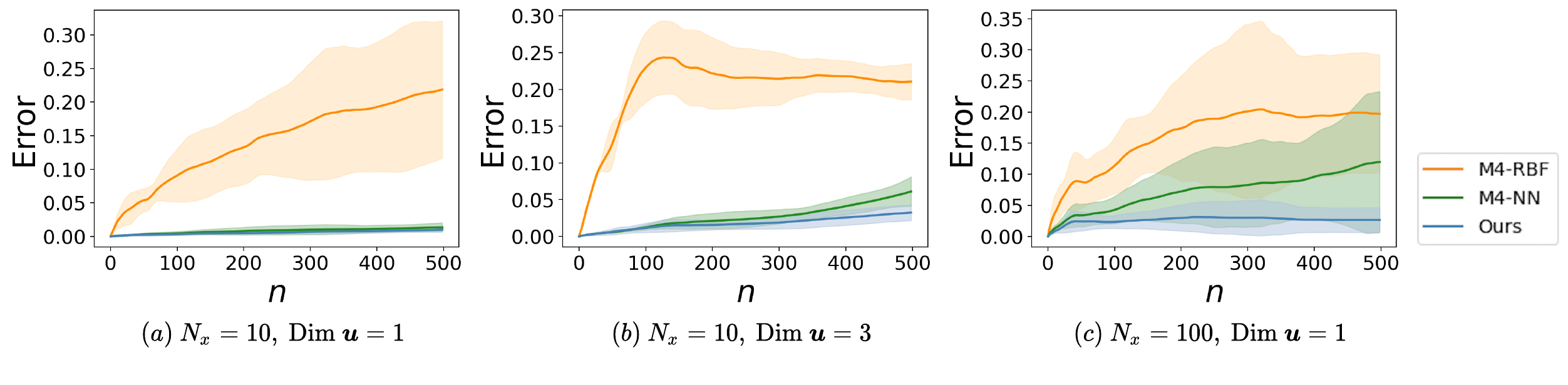}
    
        \caption{Comparison of prediction errors computed from the three approaches on the FitzHugh-Nagumo PDE discretized on space $\x$, for varying dimensions of the parameter $\uu$.
        We observe that as parameter dimension increases,
        \pk \ out-performs polynomial-based approximation schemes.}
        \label{fig:fhn_error}
\end{figure}

\begin{figure}[htb!]
    \centering
    \begin{subfigure}{1.0\textwidth}
    \centering
    \includegraphics[width=1.0\textwidth]{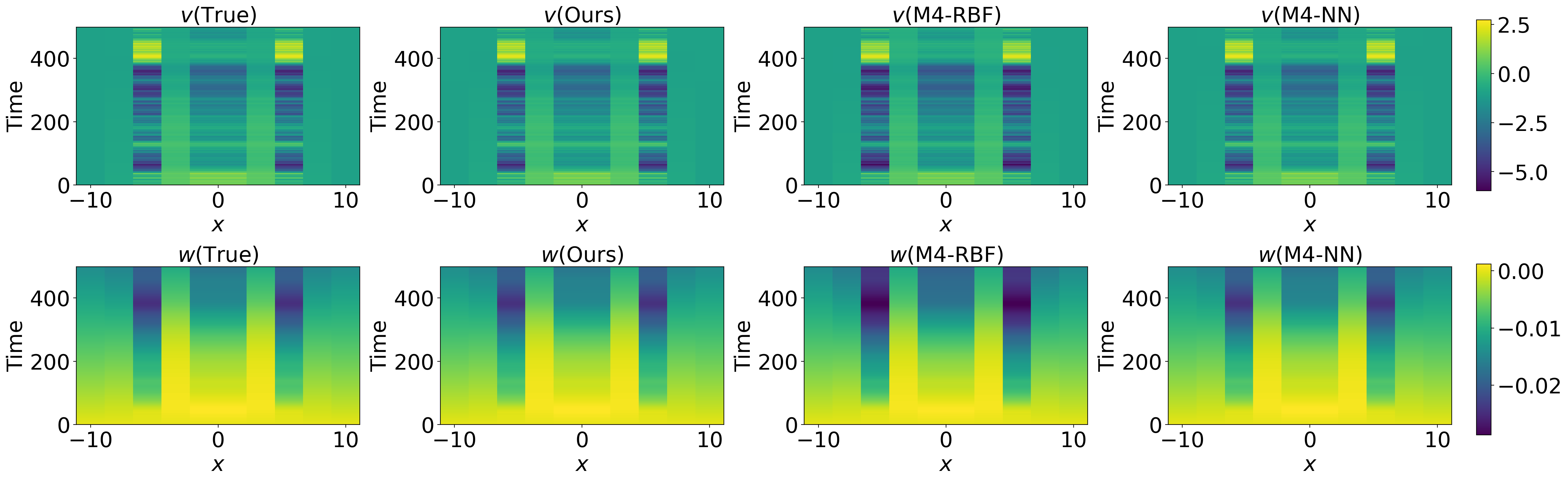}
        \caption{Trajectory.}
        \label{fig:dim_10_fhn_traj}
    \end{subfigure}
    \centering
    \begin{subfigure}{1.0\textwidth}
    \centering
    \includegraphics[width=1.0\textwidth]{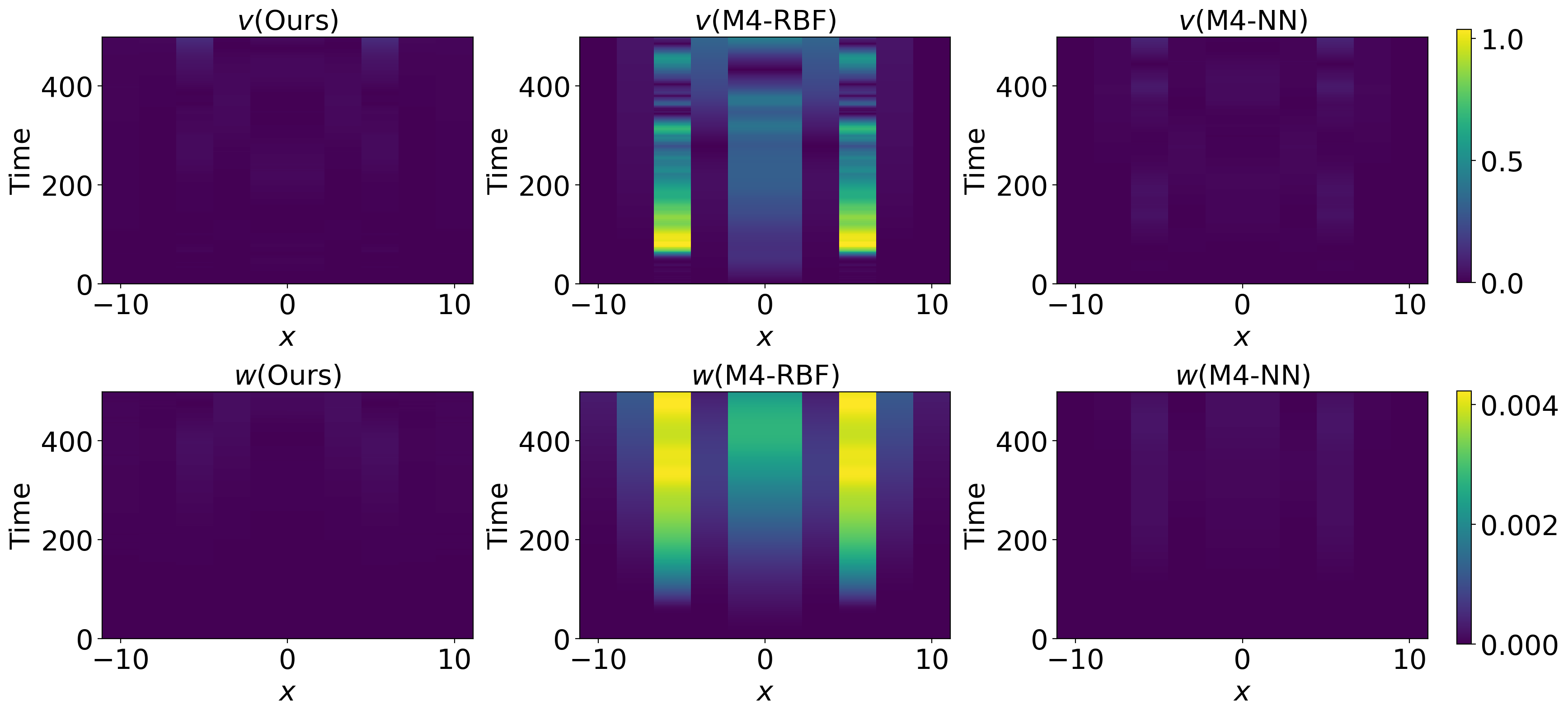}
        \caption{Difference.}
        \label{fig:dim_10_fhn_diff}
    \end{subfigure}
    \caption{(a): Prediction by PK-NN, M4-RBF and M4-NN on FHN system with $N_{\psi}=10$ and dim $\uu = 1$. (b): The absolute value of the difference between the forecasted outcomes and the ground truth.}
    \label{fig:fhn_dim_10_traj_diff}
\end{figure}

\subsection{Optimal control problems}
\label{sec:control}
As a final class of applications,
we show that \pk \ can be used to solve optimal control problems
from observational data without knowledge of the precise form of the dynamics that drive the control system.
Furthermore, it can deal with strongly non-linear control systems,
both in the state and the control.

We first apply our method to control
the forced Korteweg-De Vries (KdV) equation
(\cite{miura1976korteweg})
\begin{align}
\label{eq:kdv}
    \frac{\partial \eta(t,x)}{\partial t} + \eta(t,x) \frac{\partial \eta(t,x)}{\partial x} + \frac{\partial^{3} \eta(t,x)}{\partial x^{3}} = w(t,x),
\end{align}
where $w(t,x) = \sum_{i=1}^{3}v_{i}(x)\sin(\pi u_{i}(t))$ is the forcing term.
Control parameters at $t_{n}$ are $\uu_{n} = (u_{1,n}, u_{2,n},
u_{3,n})^{T} \in [-1,1]^{3}$.
The functions $v_{i}(x)
= e^{-25(x-c_{i})^{2}}$ are fixed spatial profiles with $c_{1} =
-\frac{\pi}{2}$, $c_{2}=0$ and $c_{3} = \frac{\pi}{2}$
\cite{korda2018linear}.
We consider periodic boundary conditions on the spatial variable $x \in [-\pi,
\pi]$ and we discretize with a spatial mesh of 128 points, and the time step is
$\Delta t = 0.01$. This induces a state-space $\boldsymbol{\eta}_{n} =
(\eta_{1}(t_{n}), \eta_{2}(t_{n}), \dots, \eta_{128}(t_{n}))^{T}$ where
$\eta_{i}(t_{n})$ is the value of $\eta(t_{n},x_{i})$ at time $t_{n}$ and
$x_{i}$ is the $i^{\text{th}}$ spatial grid point.

We consider a tracking problem involving one of the following two observables:
the ``mass'' $\int_{X} \eta(t,x) dx$ and the ``momentum'' $\int_{X}
\eta^{2}(t,x) dx$.
Given a reference trajectory (of either mass or momentum) $\{r_n\}$,
the tracking problem refers to a Bolza problem~\eqref{opt_prob_psi_control}
with $\Phi\equiv 0$ and
$L_n(m, \uu) = | m - r_{n} |^2$.
Training data are generated from 1000 trajectories of length 200 samples.
The initial conditions are a convex combination of three fixed spatial profiles
and written as $\eta(0, x) = b_{1} e^{-(x-\frac{\pi}{2})^{2}}+ b_{2}
(-\sin(\frac{x}{2})^{2}) + b_{3} e^{-(x+\frac{\pi}{2})^{2}}$ with $b_{i} > 0$
and $\sum_{i=1}^{3}b_{i} = 1$, $b_{i}$'s
are randomly sampled in $(0,1)$ with uniform distribution. The training controls $u_{i}(t)$ are uniformly randomly
generated in $[-1,1]$.
We design a common dictionary for the two tracking problems
of the form $\vpsi(\boldsymbol{\eta}) =
(1, \int_{X} \eta(t,x) dx, \int_{X} \eta^{2}(t,x) dx,
NN(\boldsymbol{\eta}))^{T}$ with 3 trainable elements
so that the resulting in the dimension of $\vpsi$ is 6.
The $NN$ is a residual network with two hidden layers having width 16 and $NN_{K}$
has two hidden layers with a width of 36.
The observable matrix $B$ in~\cref{opt_prob_psi_control}
is the row vector $(0,1,0,0,0,0)$ for mass tracking,
while $B = (0,0,1,0,0,0)$ for momentum tracking.

Instead of computing an optimal control using
non-linear programming,
we compute $\{\hat{\uu}_{n}\}_{n=0}^{N-1}$ successively
in time by solving~\cref{opt_prob_psi_control} with model predictive control (MPC)~\cite{grune2017nonlinear}.
We substitute the currently computed control into~\cref{eq:kdv},
which is integrated with RK23~\cite{bogacki19893} to get the next predicted state
$\boldsymbol{\eta}_{n+1}$.
This map is denoted by $\f$.
The predicted state is the initial value at the next MPC step.
In each MPC step,
the parametric Koopman dynamics is used to solve for control inputs over a specified time
horizon $\tau$ by minimizing the cost function~\cite{korda2018linear}.
Concretely,
to obtain $\hat{\uu}_n$
we solve the optimization problem
\begin{align}
    \label{opt_mpc}
    \begin{split}
    \min_{\{\tilde{\uu}_{i}\}_{i=n}^{n+\tau-1}} & \quad J =
    \sum_{i=n}^{n+\tau-1} \left(\|r_{i+1} - m_{i+1} \|^2 + \lambda \|\tilde{\uu}_{i}\|^{2}\right)
    \\
    \text{s.t.}
    & \quad
    m_{i+1} = B \prod_{j=0}^{i-n} K(\tilde{\uu}_{n+j}) \vpsi(\hat{\boldsymbol{\eta}}_{n}), \quad i=n,\dots,n+\tau -1,\\
    & \quad \hat{\boldsymbol{\eta}}_{n} = \f(\hat{\boldsymbol{\eta}}_{n-1}, \hat{\uu}_{n-1}), \quad \hat{\boldsymbol{\eta}}_{0} = \boldsymbol{\eta}_{0},\\
    & \quad -1 \leq u_{k,n} \leq 1, \quad k = 1,2,3,
    \end{split}
\end{align}
giving $\{\tilde{\uu}_i\}_{i=n}^{n+\tau-1}$,
and then set $\hat{\uu}_n = \tilde{\uu}_n$.
We iterate this for $n=0,1, \dots, N_{T}$ where $N_{T}$ is the number of all the time steps in the control problem.
Here, $\lambda \geq 0$ is a regularization coefficient.

In experiments, we compare three approaches: \pk,
M2~\eqref{koopman_linear} and M3~\eqref{koopman_bilinear}.
All approaches utilize trainable state dictionaries with
identical structures.
The starting values of the tracking are computed by the
state $\boldsymbol{\eta}(t_{0}) = 0.2$, then we get $m^{(\text{mass})}_{0}=1.27$
and $m^{(\text{momentum})}_{0}=0.25$. Our objective involves tracing the mass or
momentum reference trajectories
% both derived from a piece-wise function
% $\boldsymbol{\eta}(t_{n}) = 0.3$ for $n \leq 500$, and $\boldsymbol{\eta}(t_{n})
% = 0.5$ for $n > 500$. We can compute the references
% \QL{Notation: $r_n$ or $r(t_n)$?}
\begin{align*}
    \begin{aligned}
        r^{(\text{mass})}_{n}=
        \begin{cases}
            1.90, & \text{for } n \leq 500\\
            3.16, & \text{for } n > 500
        \end{cases},
        \quad
        r^{(\text{momentum})}_{n}=
        \begin{cases}
            0.57, & \text{for } n \leq 500\\
            1.58, & \text{for } n > 500
        \end{cases}.
    \end{aligned}
\end{align*}

Before discussing how to solve this control problem, we show the predictions of mass and momentum in~\cref{sec:kdv_mass_momentum_pred}. \pk \ outperforms both M2 and M3 models on the forward prediction, providing a foundation for solving control problems. Now we shift our attention back to addressing the control problems.
In~\cref{fig:kdv_comparison}, we present the tracking results
for mass (a) and momentum (b) with the MPC horizon $\tau$ set to 10.
When $\lambda = 0.005$, \pk \ exhibits better tracking capabilities
than the other two methods, attributed to its ability to capture the inherent
non-linearity present in~\cref{eq:kdv}.
Furthermore, when considering mass
tracking under $\lambda=0$,~\cref{eq:kdv} admits an explicit optimal control
solution $u_{i,n} = \frac{1}{2}$ for $0\leq n \leq 59$
and $500\leq n \leq 619$ while $u_{i,n} = 0$ otherwise.
This follows directly from integrating~\cref{eq:kdv}
over the spatial domain.
The evolution of mass under this optimal control is labelled as
``KdV'' in figure (a) of~\cref{fig:kdv_comparison}.
In the case of linear and bi-linear models,
the absence of regularization in optimization leads
to failure of mass tracking since the
optimal control takes on the values $\pm 1$, i.e. on boundaries of the control set.
Over this switching range, the control function $w(t,\cdot)$ cannot be
well approximated by a linear function.
Similarly, for momentum tracking (\cref{fig:kdv_comparison}(b)),
\pk \ outperforms the other two methods, which struggle to track the reference.

\begin{figure}[htb!]
    \centering
    \includegraphics[width=1.0\linewidth]{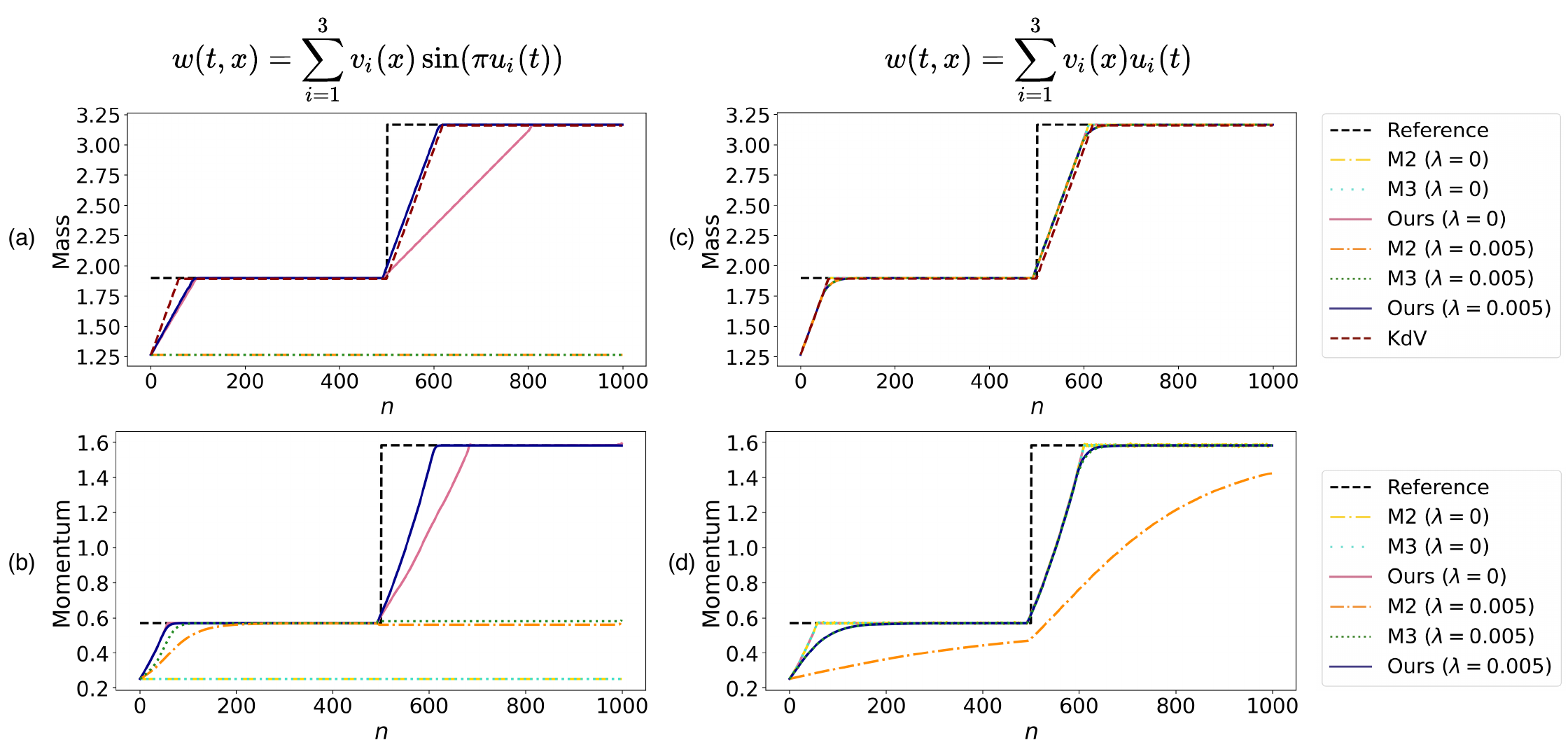}
    \caption{Tracking results for KdV equation utilizing a piece-wise constant
    reference trajectory. Subfigures (a)-(b) illustrate tracking performance for
    mass and momentum observables under $w(t,x) = \sum_{i=1}^{3}v_{i}(x)\sin(\pi
    u_{i}(t))$; Subfigures (c)-(d) present the same tracking comparison when
    $w(t,x) = \sum_{i=1}^{3}v_{i}(x) u_{i}(t)$.}
    \label{fig:kdv_comparison}
\end{figure}

To confirm that the performance disparity is due to the non-linear
dependence of the evolution equations on the control,
we perform an ablation study by setting $w(t,x)$ in~\cref{eq:kdv} to be
$\sum_{i=1}^{3}v_{i}(x)u_{i}(t)$, i.e. the forcing term depends
linearly on the control.
For $\lambda=0$, a similar explicit optimal control can be derived
($u_{i,n} = 1$ when $0 \leq n \leq 59$ and $500 \leq n \leq
619$, and $u_{i,n} = 0$ otherwise).
As expected, when the control system is linear,
the tracking outcomes of these three approaches
are similar when $\lambda = 0$,
\pk \ and M3 out-perform M2 when $\lambda = 0.005$,
as depicted in~\cref{fig:kdv_comparison}(c) and (d).
The optimization procedure employs the L-BFGS algorithm \cite{liu1989limited} using Python on a 3.20GHz Intel Xeon W-3245 with 125 GB of RAM. The average optimization times recorded are 5.82 seconds for the PK-NN model, 0.83 seconds for M2, and 1.69 seconds for M3, indicating that, when using the same optimization software, the PK-NN model requires between three and seven times longer than the other two approaches. This observation shows that \pk \ demands substantially more computational resources than the other models.
% Despite its higher computational needs,
However, \pk \ achieves satisfactory optimization results and significantly improves control capabilities. The comparative analysis of computational costs is further elaborated in~\cref{sec:kdv_computational_cost}.

We have shown that PK-NN outperforms linear (M2) and bi-linear (M3) Koopman models on tracking problems due to the stronger non-linearity. To determine if PK-NN shows improvements over other non-linear methods, we add an experiment with standard nonlinear MPC (NMPC). 
We build a neural network $\f_{NN}$ and train it on the same dataset to simulate the dynamics $\boldsymbol{\eta}_{n+1} = \f_{NN} (\boldsymbol{\eta}_{n}, \uu_{n})$. 
The NMPC approach first learns dynamics directly in the state space and then performs optimization to determine the controls,
whereas the Koopman model focuses on learning within the observable space
and controlling the observables directly.
This results in NMPC being more inefficient when the state dimension is high.
To demonstrate this, we consider problems with different spatial discretization resolutions
$N_x$.
We have discussed the results with $N_x = 128$ in our paper,
and now we extend the experiments to $N_x = 64$ and $N_x = 256$. 
The neural network $\f_{NN}$ uses a fully-connected architecture with ReLU as the activation function.
Both the input and output are vectors of size $N_x$, which matches the dimension of $\boldsymbol{\eta}$.
The parameter counts are determined by specifying a test error tolerance for forward prediction,
measured by the mean absolute error in momentum prediction.
The errors for different initial conditions are controlled to around $10^{-4}$, and multi-step predictions over a horizon length equivalent to the MPC process are controlled to below $3 \times 10^{-3}$.
We select the smallest network that can achieve this error tolerance.
The complexity of the NMPC model significantly increases as $N_x$ grows.
\Cref{table:N_x_comparison_NMPC_structure} outlines the model structures and parameter counts for PK-NN and NMPC across different $N_x$. Under these settings, we compare the tracking performance and computational cost of PK-NN and NMPC.
\begin{table}[H]
    \centering
    \small
    \begin{tabular}{c|ccc}
        \toprule
        $N_x$ & 64 & 128 &256 \\
        \midrule
        $\f_{NN}$ Arch.
         & [512,512]
         & [768,768]
         & [1024,1024] 
        \\
        \midrule
        $\f_{NN}$ Param. No.
         & 330304
         & 790400
         & 1578240
        \\
        \midrule
        \pk \ Arch. ($\vpsi$ \& $K$)
         & [16, 16] \& [36, 36]
         & [16, 16] \& [36, 36]
         & [16, 16] \& [36, 36]
        \\
        \midrule
        \pk \ Param. No.
         & 4205
         & 5229
         & 7277
        \\
        \bottomrule
    \end{tabular}
    \caption{The comparison between PK-NN and NMPC model architectures and trainable parameter counts across different $N_x$.}
    \label{table:N_x_comparison_NMPC_structure}
\end{table}
Because $\f_{NN}$ learns the dynamics of $\boldsymbol{\eta}$ not in the observable space, differing from the Koopman model's setup, the tracking problem becomes
\begin{align}
    \label{opt_nmpc}
    \begin{split}
    \min_{\{\tilde{\uu}_{i}\}_{i=n}^{n+\tau-1}} & \quad J =
    \sum_{i=n}^{n+\tau-1} \left(\|r_{i+1} - m_{i+1} \|^2 + \lambda \|\tilde{\uu}_{i}\|^{2}\right)
    \\
    \text{s.t.}
    & \quad m_{i+1} = \g(\tilde{\boldsymbol{\eta}}_{i+1}), \\
    & \quad \tilde{\boldsymbol{\eta}}_{i+1} = \f_{NN}(\tilde{\boldsymbol{\eta}}_{i}, \tilde{\uu}_{i}), \quad \tilde{\boldsymbol{\eta}}_{n} = \hat{\boldsymbol{\eta}}_{n}, \quad i=n,\dots,n+\tau -1,\\
    & \quad \hat{\boldsymbol{\eta}}_{n} = \f(\hat{\boldsymbol{\eta}}_{n-1}, \hat{\uu}_{n-1}), \quad \hat{\boldsymbol{\eta}}_{0} = \boldsymbol{\eta}_{0},\\
    & \quad -1 \leq u_{k,n} \leq 1, \quad k = 1,2,3,
    \end{split}
\end{align}
where $\g$ is the observable function corresponding to 
mass or momentum depending on the problem. 
% $\int_{X} \eta(t,x) dx, \int_{X} \eta^{2}(t,x) dx$
This optimization gives $\{\tilde{\uu}_i\}_{i=n}^{n+\tau-1}$,
and then sets $\hat{\uu}_n = \tilde{\uu}_n$.
We iterate this for $n=0,1, \dots, N_{T}$ where $N_{T}$
is the number of time steps in the control problem.
In this experiment, 
we use NMPC to track the same observable ``momentum'' as \pk \ for the problem with non-linear forcing (sin) and $\lambda = 0.005$.
The tracking results are shown in~\cref{fig:check_resolution}(c), where both \pk \ and NMPC effectively track the reference trajectory but \pk \ performs faster and more accurately. We define the mean absolute error between the tracked trajectory and the reference as our evaluation metric. The comparison of tracking errors is illustrated in~\cref{fig:check_resolution}(a), showing that \pk \ outperforms NMPC across different $N_x$, despite NMPC using a more complex neural network.
Additionally, we compare the time consumed per epoch during the training of each model. As the number of trainable parameters increases, the training cost for the NMPC model rises sharply, as demonstrated in~\cref{fig:check_resolution}(b). For the computational cost of solving the optimization problem during tracking, we compare the average, maximum, and minimum time per step for different $N_x$ in~\cref{fig:check_resolution}(d). When achieving similar tracking results, the average time cost for both training and tracking is higher for NMPC than for \pk.
These results show that as the state dimension grows, both the training and tracking challenges for NMPC increase, but \pk \ is not significantly affected. However, when the state dimension is low, NMPC can also efficiently solve the problem. Although not explicitly indicated in the figures, it is evident that with lower dimensions, there is a trend of improvement in both tracking accuracy and training time.
\begin{figure}[htb!]
    \centering
    \includegraphics[width=1.0\linewidth]{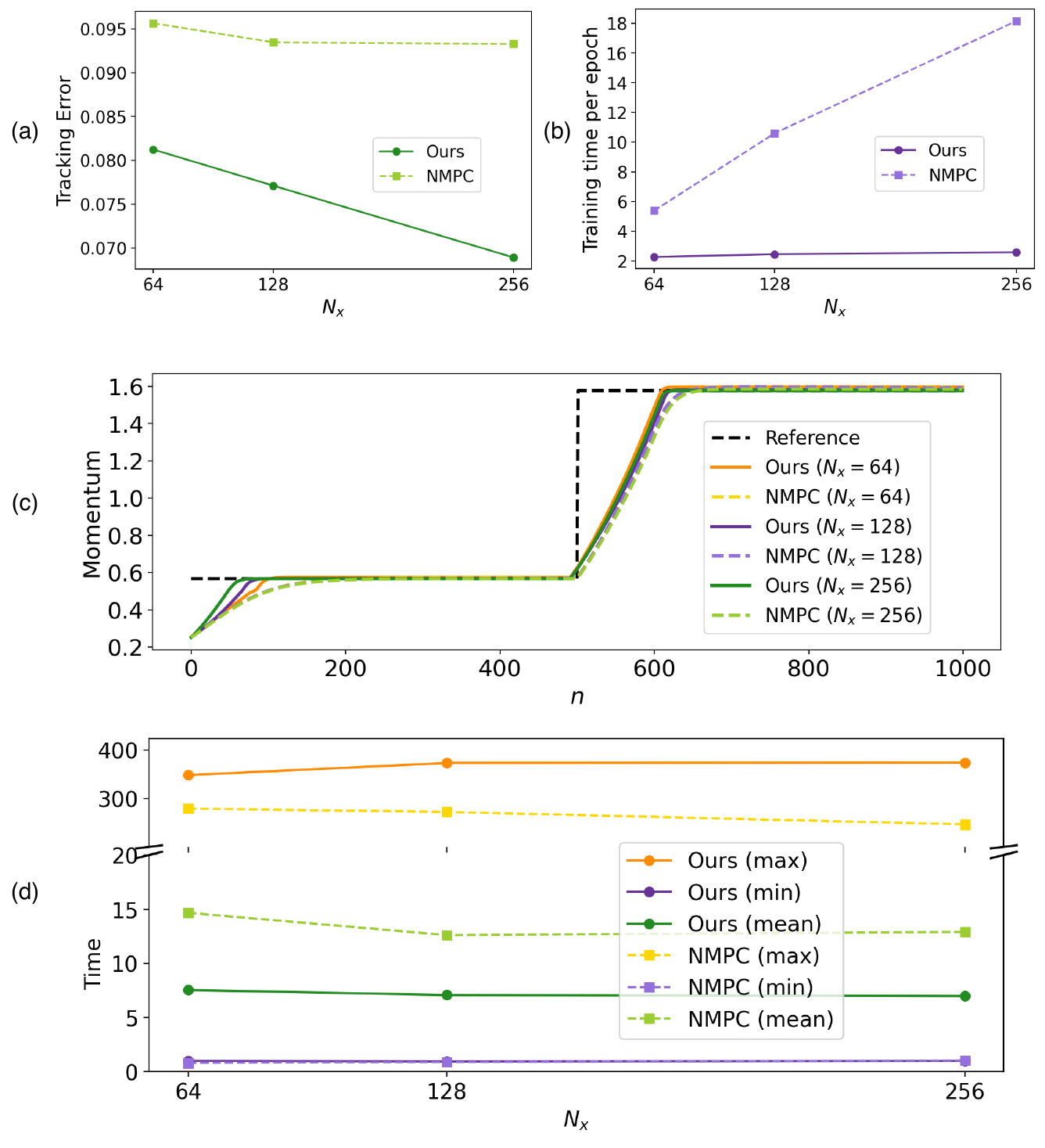}
    \caption{Comparison between PK-NN and NMPC. 
    (a) Tracking error results for both methods under different $N_x$. 
    (b) Average time per epoch (in seconds) during model training for different 
    $N_x$. 
    (c) Tracking results for cases with varying $N_x$. 
    (d) Time consumption per optimization step for solving tracking problems with different $N_x$, showing the average, maximum, and minimum values.}
    \label{fig:check_resolution}
\end{figure}

We end with an analysis of the controllability of the learned
parametric Koopman dynamics as discussed in~\cref{sec:controllability}.
Recall that a sufficient condition for controllability
is for the matrix $C$ (see~\cref{eq:matrix_C}) to have full rank.
When using a deep, fully connected network
to approximate $K$, we can control the rank
of $C$ by setting $N_K$, the number of learned features
(the width) of the penultimate layer.
In the following, we repeat the previous experiments
for $N_{K}=37, 31, 13, 7$.
After training, we uniformly randomly sample 2000 $\uu_{i} \in U$ to compute the
matrix $C$.
The decay of the singular values of $C$ is shown in~\cref{fig:singular_value_NK}.
Note that the dimension of the dictionary is $6$, and the components in the first row of the matrix $K(\uu)$ are constants. Hence, the rank condition in this setting reduces to $5 \times 6 = 30$.
Only the case $N_{K}=37$ satisfies the rank condition
that guarantees controllability.
We now solve the tracking problem under the three choices of $N_K$
for $\lambda=0.005$,
and the results shown in~\cref{fig:tracking_mass_NK} and~\cref{fig:tracking_momentum_NK}
confirms that the tracking performance may improve with greater controllability.
A detailed discussion can be found in~\cref{sec:K_structure}.

\begin{figure}[htb!]
    \centering
        \begin{subfigure}{0.32\textwidth}
            \centering
            \includegraphics[width=1.0\linewidth]{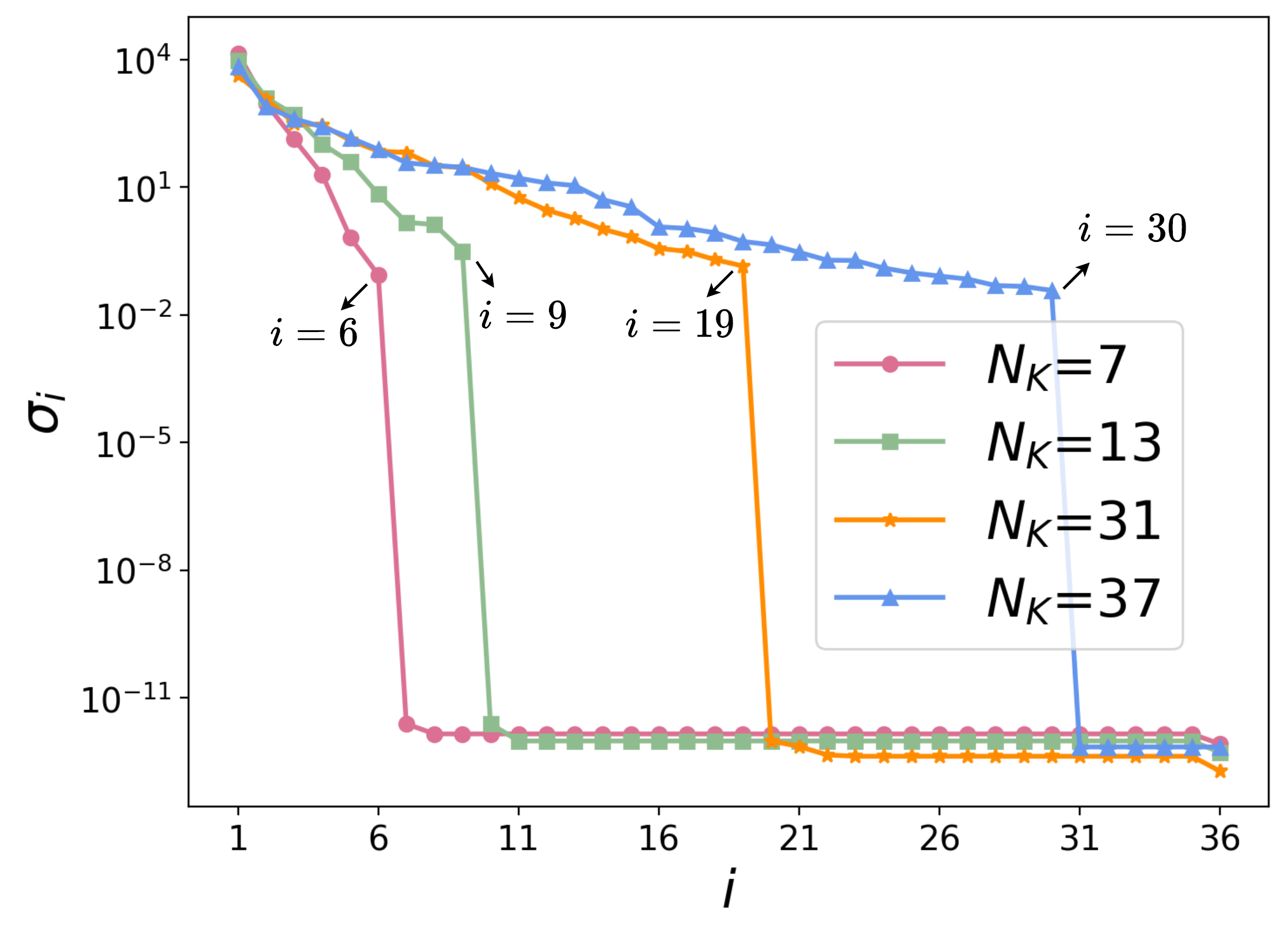}
            \caption{Singular values of $C$.}
                \label{fig:singular_value_NK}
        \end{subfigure}
        \begin{subfigure}{0.32\textwidth}
            \centering
            \includegraphics[width=1.0\linewidth]{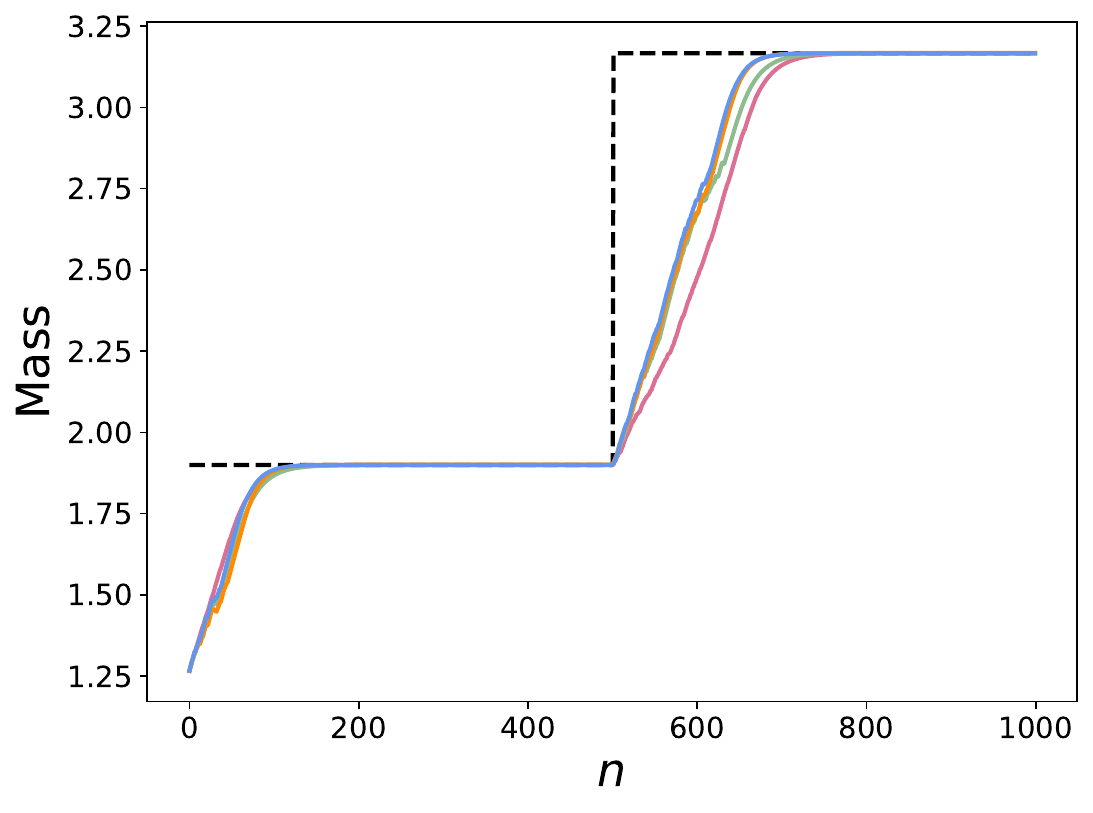}
            \caption{Track mass.}
      \label{fig:tracking_mass_NK}
        \end{subfigure}
        \begin{subfigure}{0.32\textwidth}
            \centering
            \includegraphics[width=1.0\linewidth]{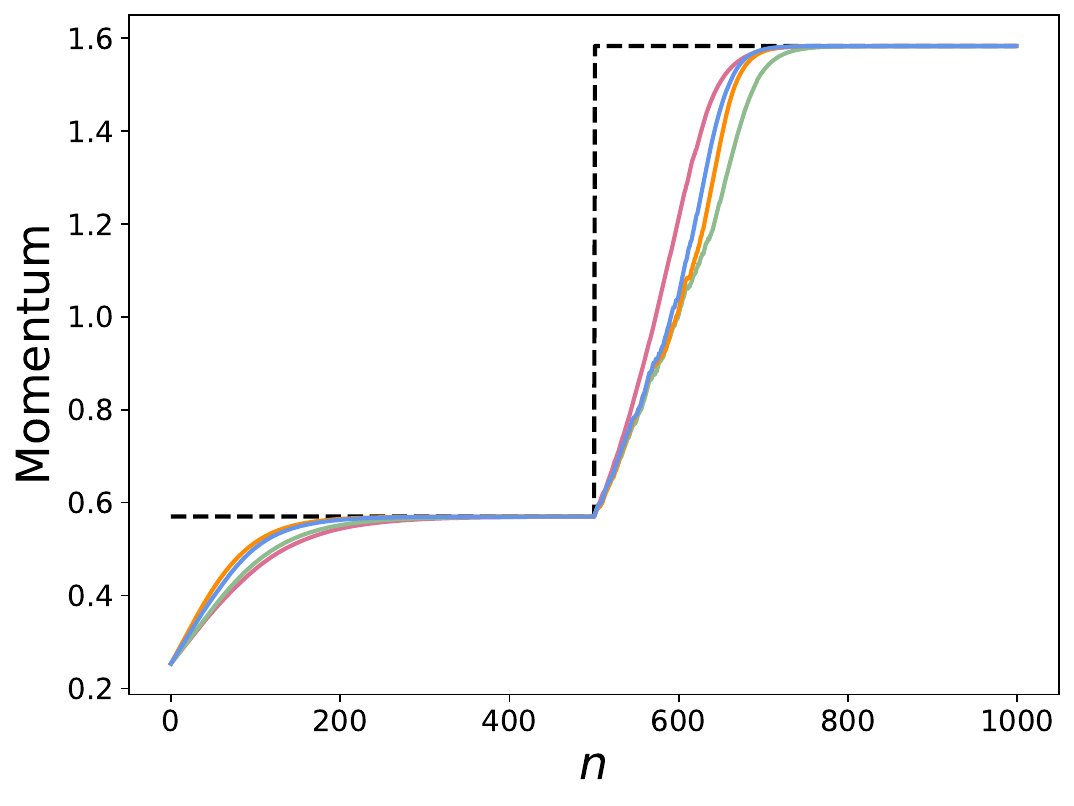}
            \caption{Track momentum.}
      \label{fig:tracking_momentum_NK}
        \end{subfigure}
        \caption{Comparison on the controllability of \pk \ with $N_{K} = 7, 13, 37$.}
        \label{fig:comparison_NK}
    \end{figure}

\section{Conclusion}
We propose and implement a data-driven methodology for creating approximate Koopman dynamics for parametric dynamical systems. The primary challenge with the Koopman operator lies in identifying the invariant subspace that is independent of parameters and approximating a projected Koopman operator that acts on this space. To address this, we use neural networks to parameterize both the parameter-dependent Koopman operator and the parameter-independent state dictionary that spans an invariant subspace. These are trained simultaneously using trajectory data, with a loss function designed to reflect the dynamics within the Koopman latent space. 

In prediction tasks, \pk \ achieves improved performance over current methods, including optimized DMD, EDMD with RBF dictionary and EDMD with dictionary learning. Our method performs well in handling highly non-linear problems, as evidenced by experiments with the Van der Pol Mathieu equations, surpassing linear and bilinear variants of the existing approach for parametric, or input-dependent, Koopman operator analysis. Additionally, \pk \ performs more effectively in scenarios involving high-dimensional states and parameters, outperforming existing Koopman structures generated by polynomial bases, with the FitzHugh-Nagumo (FhN) equation serving as a test case.
In addition to prediction problems, in control applications our approach surpasses the performance of 
existing Koopman operator analysis with linear or bilinear control schemes,
particularly in tracking observables in the Korteweg-de Vries (KdV) equation.
This adaptability to high-dimensional and strongly non-linear problems makes
our approach a suitable candidate for large-scale data-driven prediction and control applications.

\section*{Acknowledgements}
The authors acknowledge discussions with
MO Williams and Boyuan Liang at the initial stages of this work.

\newpage
\appendix

\section{Proof of~\cref{prop:parametric composition operator}}
\label{proof:Continuous Operators}
\begin{proof}
For each $\uu$, we suppose $\Kcal(\uu)$ is the composition operator induced by $\f(\cdot, \uu)$. We consider $S \in \Scal$ such that $m(S) < \infty$, then the indicator function $\mathbf{1}_{S} \in L^{2}(X, m)$.
If $\x \in \f^{-1}(\cdot, \uu)(S)$, we have $\f(\x, \uu) \in S$ since $\f(\cdot, \uu)$ is non-singular.
It is equivalent to $(\mathbf{1}_{S} \circ \f(\cdot,\uu))(\x) = 1$, which induces $\Kcal(\uu)\mathbf{1}_{S}(\x) = 1$.
Thus, we have $\Kcal(\uu)\mathbf{1}_{S} = \mathbf{1}_{\f^{-1}(\cdot, \uu)(S)}$.
Then we have
\begin{align}
    \begin{split}
        m(\f^{-1}(\cdot, \uu)(S)) &= \int \mathbf{1}_{\f^{-1}(\cdot, \uu)(S)} dm\\
        &= \int \Kcal(\uu)\mathbf{1}_{S} dm\\
        &= \int |\Kcal(\uu)\mathbf{1}_{S}|^{2} dm\\
        &= \|\Kcal(\uu)\mathbf{1}_{S}\|^{2}\\
        &\leq \|\Kcal(\uu)\|^{2} \|\mathbf{1}_{S}\|^{2}\\
        &= \|\Kcal(\uu)\|^{2} m(S),
    \end{split}
\end{align}
where $\|\Kcal\| = \sup_{\|\phi\|=1}\|\Kcal \phi\|$.
Let $b_{\uu} = \|\Kcal(\uu)\|^{2}$. Then we have $m(\f(\cdot, \uu)^{-1}(S)) \leq b_{\uu} m(S)$.

Conversely, if we have $m(\f^{-1}(\cdot, \uu)(S)) \leq b_{\uu} m(S)$,
then $m \f^{-1}(\cdot, \uu) \ll m$ (``$\ll$'' means absolute continuity).
The Radon-Nikodym derivative
$h_{\f(\cdot, \uu)}$ of $m\f^{-1}(\cdot, \uu)$ with respect to $m$
exists and $h_{\f(\cdot, \uu)} \leq b_{\uu}$ almost everywhere.
Let $\phi \in L^{2}(X, m)$ and then we have
\begin{align}
        \|\Kcal(\uu) \phi\|^{2}
        = \int |\phi \circ \f(\cdot, \uu)|^{2} dm
        = \int |\phi|^{2} d m\f^{-1}(\cdot, \uu)
        = \int |\phi|^{2} h_{\f(\cdot, \uu)} dm
        \leq b_{\uu} \|\phi\|^{2}.
\end{align}
This shows that $\Kcal(\uu)$ is a bounded (and continuous) operator on $L^{2}(X, m)$.
\end{proof}

%%%%%%%%%%%%%%%%%%%%%%%%%%%%%%%%%%%%%

\section{Proof of \cref{prop:finite_param_koopman}}
\label{proof:finite_param_koopman}
In this section, we prove~\cref{prop:finite_param_koopman}. The following~\cref{lemma:pN_convergence} and~\cref{lemma:continuity_on_u} are necessary for our discussion.
\begin{lemma}
    \label{lemma:pN_convergence}
    Let $X$ be a finite measurable space with measure $m$.
    If $(e_{i})_{i=1}^{\infty}$ form an orthonormal basis of space $C_b(X)$. Define $P_{N}g = \argmin_{\tilde{g} \in \mathrm{span}((e_{i})_{i=1}^{\infty})} \|\tilde{g}-g\|$, then $P_{N}$ converge strongly to the identity operator $I$.
\end{lemma}

\begin{proof}
    Let $\phi = \sum_{i=1}^{\infty}c_{i}e_{i}$ with $\|\phi\|=1$. Then by Parseval's identity $\sum_{i=1}^{\infty}|c_{i}|^{2} = 1$ and
    \begin{align}
        \|P_{N}\phi - \phi\|^{2} = \|
        \sum_{i=N+1}^{\infty} c_{i}e_{i}\|^{2}
        =
        \sum_{i=N+1}^{\infty}| c_{i}|^{2}
        \to 0, \quad \text{as } N \to \infty.
    \end{align}
\end{proof}

\begin{lemma}
    \label{lemma:continuity_on_u}
    Let $X$ be a finite measurable space with measure $m$.
    Given an observable function $\phi \in C_{b}(X)$. Assume that $f(\x,\cdot): U \to X$ is continuous for almost all $\x \in X$. If (\ref{eq:pkoopman_ref}) holds, then $\Kcal(\cdot)\phi: U \to C_b(X)$ is continuous on $U$ with respect to $L^{2}$ norm, in the sense that $\lim_{\uu_{n}\to\uu} \|\Kcal(\uu_{n})\phi - \Kcal(\uu)\phi\| = 0$.
\end{lemma}

\begin{proof}
Since $\f(\x,\uu_{n})\to\f(\x,\uu)$ as $\uu_{n} \to \uu$ almost everywhere $\x \in X$ and $\phi(\x)$ is a continuous function, we have
\begin{align}
\label{limit_phi_f}
    \lim_{\uu_{n}\to\uu} |\phi\circ\f(\x,\uu_{n}) - \phi\circ\f(\x,\uu)|^{2} = 0
\end{align}
almost everywhere on $X$.
Since $|\phi(\x)|\leq M$ on $X$ for some $M \in [0,\infty)$, we have $|\phi\circ\f(\x,\uu_{n}) - \phi\circ\f(\x,\uu)|^{2} \leq 4 M^{2}$. By the dominated convergence theorem,
\begin{align}
    \label{limit_K(u)_phi}
    \begin{split}
    \lim_{\uu_{n}\to\uu} \|\Kcal(\uu_{n})\phi - \Kcal(\uu)\phi\|^{2} &= \lim_{\uu_{n}\to\uu} \int_{X}|\Kcal(\uu_{n})\phi(\x) - \Kcal(\uu)\phi(\x)|^{2}m(d\x)\\
    & = \lim_{\uu_{n}\to\uu} \int_{X}|\phi\circ\f(\x,\uu_{n}) - \phi\circ\f(\x,\uu)|^{2}m(d\x)\\
    & = \int_{X} \lim_{\uu_{n}\to\uu}|\phi\circ\f(\x,\uu_{n}) - \phi\circ\f(\x,\uu)|^{2}m(d\x)\\
    & = 0.
    \end{split}
    \end{align}
    Therefore, $\lim_{\uu_{n}\to\uu} \|\Kcal(\uu_{n})\phi - \Kcal(\uu)\phi\| = 0$.
\end{proof}

In~\cref{prop:ineq_fininte_param}, we demonstrate that an inequality employed in the proof of~\cref{prop:finite_param_koopman} is applicable to a finite set of parameters for given observables.

\begin{proposition}
    \label{prop:ineq_fininte_param}
    Let $U = \{\uu_{1}, \uu_{2}, \dots, \uu_{s}\}$ be a finite set of parameters
    and assume that the observables of interest
    $\g$ satisfy $g_j \in C_b(X)$.
    Then, for any $\varepsilon > 0$,
    there exists a positive integer $N_\psi > 0$,
    a dictionary $\vpsi = \{\psi_1,\dots,\psi_{N_\psi}\}$
    with $\psi_i \in L^2(X,m)$,
    a set of vectors $\{\va_{j} \in \R^{N_{\psi}}: j=1,\dots,N_y\}$,
    such that $g_j = \va_j^T \vpsi$ and for all $\uu \in U$,
    we have
    $\|\Kcal(\uu)g_{j} - P_{N_{\psi}}\Kcal(\uu) g_{j}\|\leq \varepsilon$, where $P_{N_\psi}g = \argmin_{\tilde{g} \in \mathrm{span}(\vpsi_{N_{\psi}})} \|\tilde{g}-g\|$.
\end{proposition}

\begin{proof}
    Let $\psi_{j} = g_{j}, j=1,\dots,N_{y}$ and this can ensure the existence of $\va_{j}$. Let $(\psi_{i})_{i=1}^{\infty}$ be a sequence of functions in $C_b(X)$ such that $\mathrm{span}((\psi_{i})_{i=1}^{\infty})$ is dense in $C_b(X)$. The dictionary $\vpsi_{N} =  (\psi_1(\x), \psi_2(\x), \dots, \psi_{N}(\x))^T$ consists of the first $N$ elements of $(\psi_{i})_{i=1}^{\infty}$. Then we can compute an orthonormal basis $(e_{i})_{i=1}^{N}$ of $\mathrm{span}(\vpsi_{N})$ by Gram-Schmidt process and note $(e_{i})_{i=1}^{\infty}$ as an orthonormal basis of $C_b(X)$ space. Thus, $\mathrm{span}(\vpsi_{N}) = \mathrm{span}((e_{i})_{i=1}^{N})$ and $P_{N}g = \argmin_{\tilde{g} \in \mathrm{span}((e_{i})_{i=1}^{N})} \|\tilde{g}-g\|$ by definition.

    For each $\uu_{i} \in U$, we have $\x_{n+1} = \f(\x_{n}, \uu_{i})$ and $g_{j}(\x_{n+1}) = \Kcal(\uu_{i})g_{j}(\x_{n})$.
    By~\cref{lemma:pN_convergence}, for any $ \varepsilon_{i} >0$, there exists $ N_{i,j}^{*} \in \N^{+}$ such that $N_{i,j}\geq N_{i,j}^{*}$ implies
    \begin{align}
        \frac{\|\Kcal(\uu_{i})g_{j} - P_{N_{i,j}}\Kcal(\uu_{i}) g_{j} \|} {\|\Kcal(\uu_{i}) g_{j}\|}
        = \| \frac{\Kcal(\uu_{i})g_{j}}{\|\Kcal(\uu_{i}) g_{j}\|} -  P_{N_{i,j}}\frac{\Kcal(\uu_{i}) g_{j}}{\|\Kcal(\uu_{i}) g_{j}\|}\|
        \leq \varepsilon_{i}
    \end{align}
    for each $j$.
    Let $N_{i} = \max\{N_{i,j}^{*}\}_{j=1,\dots,N_y}$.
    Then we note that
    \begin{align}
    \label{ineq:fininte_param_separately}
    \|\Kcal(\uu_{i})g_{j} - P_{N_{i}}\Kcal(\uu_i) g_{j} \|
    =\frac{\|\Kcal(\uu_{i})g_{j} - P_{N_{i}}\Kcal(\uu_{i}) g_{j} \|} {\|\Kcal(\uu_{i}) g_{j}\|} \|\Kcal(\uu_{i}) g_{j}\|
    \leq \varepsilon_{i} \sqrt{b_{\uu_{i}}} \|g_{j}\|.
    \end{align}
    Given $\varepsilon > 0$, we set $\varepsilon_{i} = \frac{\varepsilon}{\sqrt{b_{\uu_{i}}}\max\{\|g_{j}\|\}}$ and $b_{\uu_{i}}$ is the coefficient discussed in~\cref{prop:parametric composition operator}. Then for each $\uu_{i}$, there exists a corresponding $N_{i}$ such that the inequality (\ref{ineq:fininte_param_separately}) holds.
    Define $N_{\psi} = \max{\{N_{1}, N_{2}, \dots, N_{n}\}}$, then
    \begin{align}
    \|\Kcal(\uu_{i})g_{j} - P_{N_{\psi}}\Kcal(\uu_i) g_{j} \|  \leq \varepsilon
    \end{align}
    holds for all $\uu_{i} \in U_{n}$ and all the given $g_{j}$'s.

\end{proof}

Now we have the \textbf{proof of~\cref{prop:finite_param_koopman}}.

\begin{proof}
\textbf{Step 1:}
By~\cref{lemma:continuity_on_u},
% the mapping $\Kcal(\cdot): U \to \mathcal{B}(C_b(X))$ is continuous on $U$, then 
for each given observable $g_{j}$, $\Kcal(\cdot)g_{j} : U \to C_b(X)$ is uniformly continuous on compact set $U$.
We know that for any $\varepsilon '>0$, $\exists \delta_{j} >0$ such that $\|\Kcal(\vv)g_{j} - \Kcal(\uu)g_{j}\| \leq \varepsilon '$ for any $\|\vv - \uu\|\leq \delta_{j}$.
Let $\delta = \min\{\delta_{j}\}$.
We introduce the notion of a localized neighbourhood around each point $\vv_{i}$ in the set $U$.
This neighbourhood is formally defined as
\begin{align}
    B_{i} = \{\uu \in U: \|\vv_{i} - \uu\| \leq \delta, \vv_{i} \in U\},
\end{align}
where $\delta$ is a predetermined positive radius.
Given the collection $\{B_{i}\}$, it suffices to form an open cover for $U$.
By the compactness of $U$, we can assert the existence of a finite subcover, denoted as $\{B_{1}, B_{2}, \dots, B_{n^*}\}$, which is sufficient to entirely cover the set $U$.

For each $g_{j}$, we denote
$
    A_{i,j} =  \Kcal(B_{i})g_{j}
$
as the image of $B_{i}$ by $\Kcal(\cdot)g_{j}$. Then we have $\| \Kcal(\vv_{i})g_{j} - g'\| \leq \varepsilon '$ for any $g' \in A_{i,j}$ and the collection $\{A_{1,j}, A_{2,j}, \dots, A_{n^*,j}\}$ can cover $A_{j}=\Kcal(U)
g_{j}$.
Naturally, for any $\uu \in B_{i}$, we have $\|\vv_{i} - \uu\| \leq \delta$ and
\begin{align}
\label{ineq:uniform_continuity}
\|\Kcal(\vv_{i})g_{j} - \Kcal(\uu)g_{j} \| \leq \varepsilon '
\end{align}
for all given $g_{j}$'s.

\textbf{Step 2:}
Let us consider how to get the dimension $N_{\psi}$ of the dictionary $\vpsi$
by~\cref{prop:ineq_fininte_param}.
For any $\varepsilon ' > 0$, $N_{\psi}$ can be determined by $\{\vv_{i}\}_{i=1}^{n^*}$ which are the centers of $\{B_{i}\}_{i=1}^{n^*}$, such that
\begin{align}
    \label{ineq:fininte_param}
    \|\Kcal(\vv_{i})g_{j} -P_{N_{\psi}}\Kcal(\vv_{i}) g_{j} \|\leq \varepsilon '
\end{align}
for any $\vv_{i} \in \{\vv_{i}\}_{i=1}^{n^*}$ and all given $g_{j}$'s.

\textbf{Step 3:}
Follow step 1, for any $\uu \in B_{i}$ and all given $g_{j}$'s, the corresponding $P_{N_{\psi}}$ can give
\begin{align}
\label{ineq:uniform_continuity_pN}
\begin{split}
    \|P_{N_{\psi}}\Kcal(\vv_{i})g_{j} - P_{N_{\psi}}\Kcal(\uu)g_{j} \|
    &= \|P_{N_{\psi}}( \Kcal(\vv_{i})g_{j}- \Kcal(\uu)g_{j})\|\\
    &\leq \|P_{N_{\psi}}\|\|\Kcal(\vv_{i})g_{j} - \Kcal(\uu)g_{j}\|\\
    &\leq \varepsilon '
\end{split}
\end{align}
since $\|P_{N_{\psi}}\| \leq 1$.

\textbf{Step 4:}
We consider the analysis for any $\uu \in U$, then $\uu$ must belong to at least one of the neighbourhoods $\{B_{i}\}_{i=1}^{n^*}$. If $\uu \in B_{i}$, we consider (\ref{ineq:fininte_param}), (\ref{ineq:uniform_continuity}) and (\ref{ineq:uniform_continuity_pN}) and obtain
\begin{align}
\begin{split}
    &\| \Kcal(\uu)g_{j} - P_{N_{\psi}}\Kcal(\uu)g_{j}\|\\
    \leq&
    \|\Kcal(\uu)g_{j} - \Kcal(\vv_{i})g_{j}\|
    + \|\Kcal(\vv_{i})g_{j} - P_{N_{\psi}}\Kcal(\vv_{i})g_{j}\|
    + \| P_{N_{\psi}}\Kcal(\vv_{i})g_{j} - P_{N_{\psi}}\Kcal(\uu)g_{j}\|\\
    \leq&  \varepsilon,
\end{split}
\end{align}
when we set $\varepsilon ' = \frac{\varepsilon}{3}$.

\textbf{Step 5:}
Notice that for any $\uu \in U$, there exists a matrix $K(\uu) \in \R^{N_{\psi} \times N_{\psi}}$ such that
\begin{align}
P_{N_{\psi}}\Kcal(\uu)g_{j} = \va_{j}^{T}K(\uu)\vpsi ,
\end{align}
where $\vpsi$ is defined as discussed in the proof of~\cref{prop:ineq_fininte_param}.
Since $\Kcal(\cdot)g_{j}: U \to C_b(X)$ is continuous on $U$, then $\va_{j}^{T}K(\cdot)\vpsi: U \to C_b(X)$ is continuous on $U$ for $j = 1,2,\dots, N_{y}$. Thus, the entries in $K$ are continuous on $U$. Then we note that
$K: U \to \R^{N_{\psi}\times N_{\psi}}$ is a continuous function with respect to Frobenius norm and
\begin{align}
    \| \Kcal(\uu)g_{j} -\va_{j}^{T}K(\uu)\vpsi \| =
    \| \Kcal(\uu)g_{j} - P_{N_{\psi}}\Kcal(\uu)g_{j}\| \leq \varepsilon.
\end{align}
\end{proof}

\section{\texorpdfstring{Neural network structure of \( K(u) \)}{Neural network structure of K(u)}}

\label{sec:K_structure}

In this section, we explore the relationship between the number of nodes in the last hidden layer of neural network structure for $K(\uu)$ and the rank of matrix $C$. The structure of $K(\uu)$ is illustrated in~\cref{fig:K_structure} and the matrix $C$ is discussed in~\cref{sec:application_analysis}. We explain this relationship by focusing on the KdV case detailed in~\cref{sec:control}. In our setting, $d=N_{\psi}$. The output layer in our neural network configuration for $K(\uu)$ is a dense layer that includes a bias term. Therefore, if $N_{(\text{last hidden})}$ is the number of nodes in the last hidden layer, $N_{K}$ should be $N_{(\text{last hidden})}+1$. We can regard the flattened matrix $K(\uu)$ as the vector obtained prior to the ``reshape'' operation. Each element in this vector is computed by $N_{K}$ basis functions. Consequently, the rank of matrix $C$ is determined by $\min\{N_{K}, d^{2}\}$, always equal to or less than this number.
For the case we introduce in~\cref{sec:control}, $d = 6$ and there exists 6 constants in the first row of $K(\uu)$, neural networks are used to generated the other entries in $K(\uu)$. The rank of matrix $C$ is less than or equal to $\min\{N_{K}, d*(d-1)\}$,

\begin{figure}[htb!]
	\centering
	\includegraphics[width=0.8\textwidth]{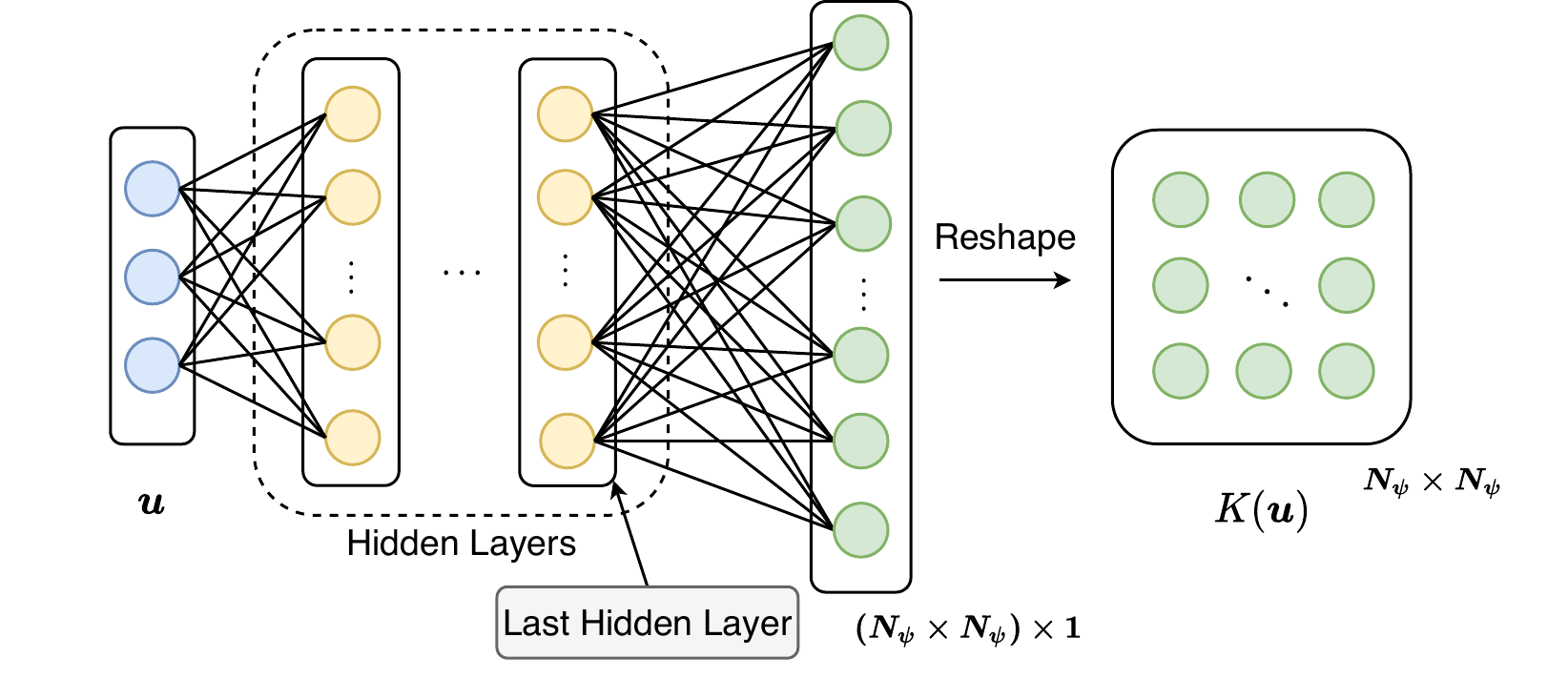}
	\caption{The Neural Networks architecture of $K(\uu)$.}
	\label{fig:K_structure}
\end{figure}

\section{Analysis on Van der Pol Matheiu case with low non-linearity}
\label{sec:vdpm_mu0}
In the case where the parameter $\mu = 0$, the PK-NN model does not perform as well as models M2 and M3. However, the structure of PK-NN can encompass the capabilities of both M2 and M3 models. To understand the reason why PK-NN is outperformed, its structure is set to include a ``hardtanh'' (originally ``tanh'' in the experiment before) activation function in the hidden layer, and it initially adopts the same weights and biases for the output layer as the successful model M2. The initial setup is slightly modified (perturbed) from this configuration and \pk \ is retrained. The outcomes reveal that the performance of PK-NN, which is indicated by a purple curve becomes similar to that of M2 with this precise initialization. Additionally, the blue curve shows the performance of PK-NN with the same ``hardtanh'' activation function but starting from a random initialization, highlighting that PK-NN struggles to learn effectively from scratch but can achieve good results with a suitable initialization. The results are shown in~\cref{fig:vdp_mu0}. This finding suggests that the primary challenge with PK-NN is its trainability rather than its fundamental design.

\begin{figure}[H]
    \centering
    \begin{subfigure}{0.48\textwidth}
    \centering
    \includegraphics[width=1.0\textwidth]{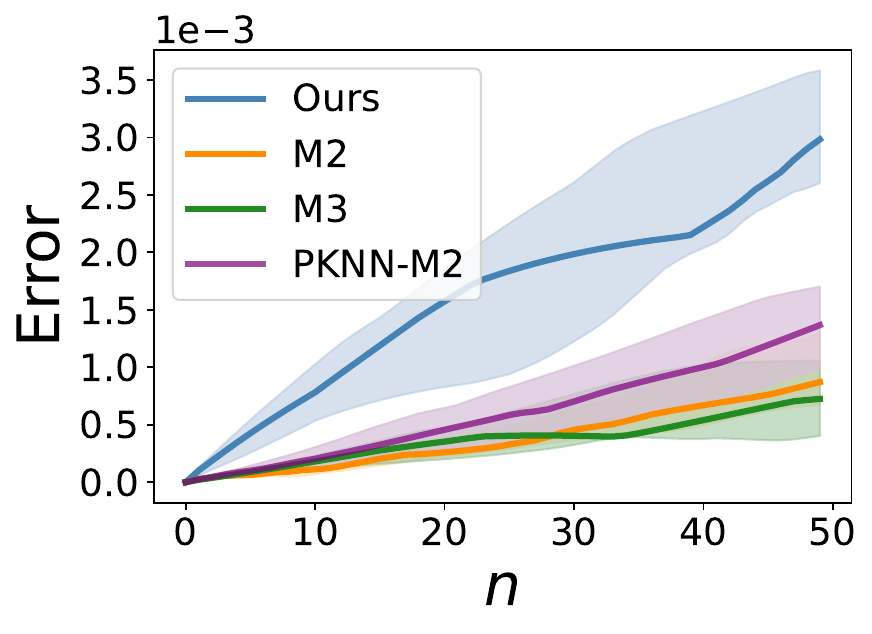}
    \caption{Error.}
    \label{fig:error_pknn_linear_perturb}
    \end{subfigure}
    \begin{subfigure}{0.48\textwidth}
        \centering
        \includegraphics[width=1.0\textwidth]{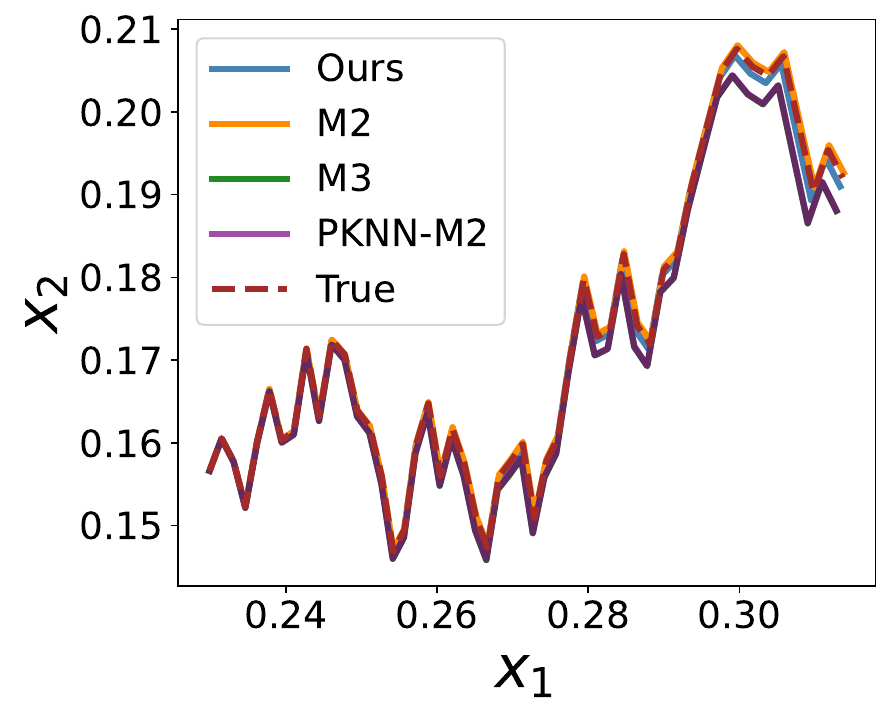}
    \caption{Trajectory.}
    \label{fig:traj_pknn_linear_perturb}
    \end{subfigure}
    \caption{The performance of PK-NN with perturbation on a precise initialization, which is the same as M2 setting, shown as PKNN-M2.}
    \label{fig:vdp_mu0}
\end{figure}

\section{Trajectory results in original space for the FitzHugh-Nagumo system}
\label{sec:fhn_traj}
Here are the results about trajectories on the original two-dimensional space of FitzHugh-Nagumo system and the difference between the prediction and the ground truth. The results in~\cref{fig:fhn_dim_10_traj_diff} is about the case with $N_{x} = 10$ and dim $\uu = 1$.
\cref{fig:fhn_dim_10_high_u_traj_diff} and~\cref{fig:fhn_dim_100_traj_diff} show for different spatial discretisations ($N_{x} = 100$, dim $\uu = 1$) and control dimensions($N_{x} = 10$, dim $\uu = 3$). We can see the performance among PK-NN, M4-RBF and M4-NN directly. Our \pk \ outperforms the other two models in all cases.

\begin{figure}[H]
    \centering
    \begin{subfigure}{1.0\textwidth}
    \centering
    \includegraphics[width=1.0\textwidth]{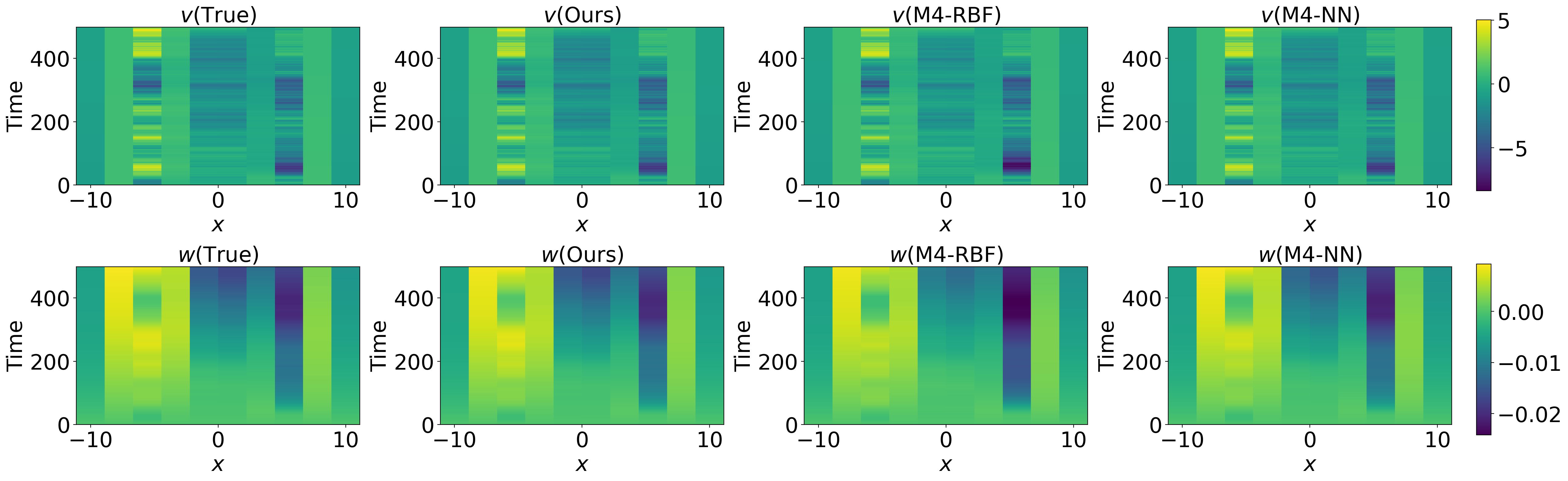}
        \caption{Trajectory.}
        \label{fig:dim_10_high_u_fhn_traj}
    \end{subfigure}
    \centering
    \begin{subfigure}{1.0\textwidth}
    \centering
    \includegraphics[width=1.0\textwidth]{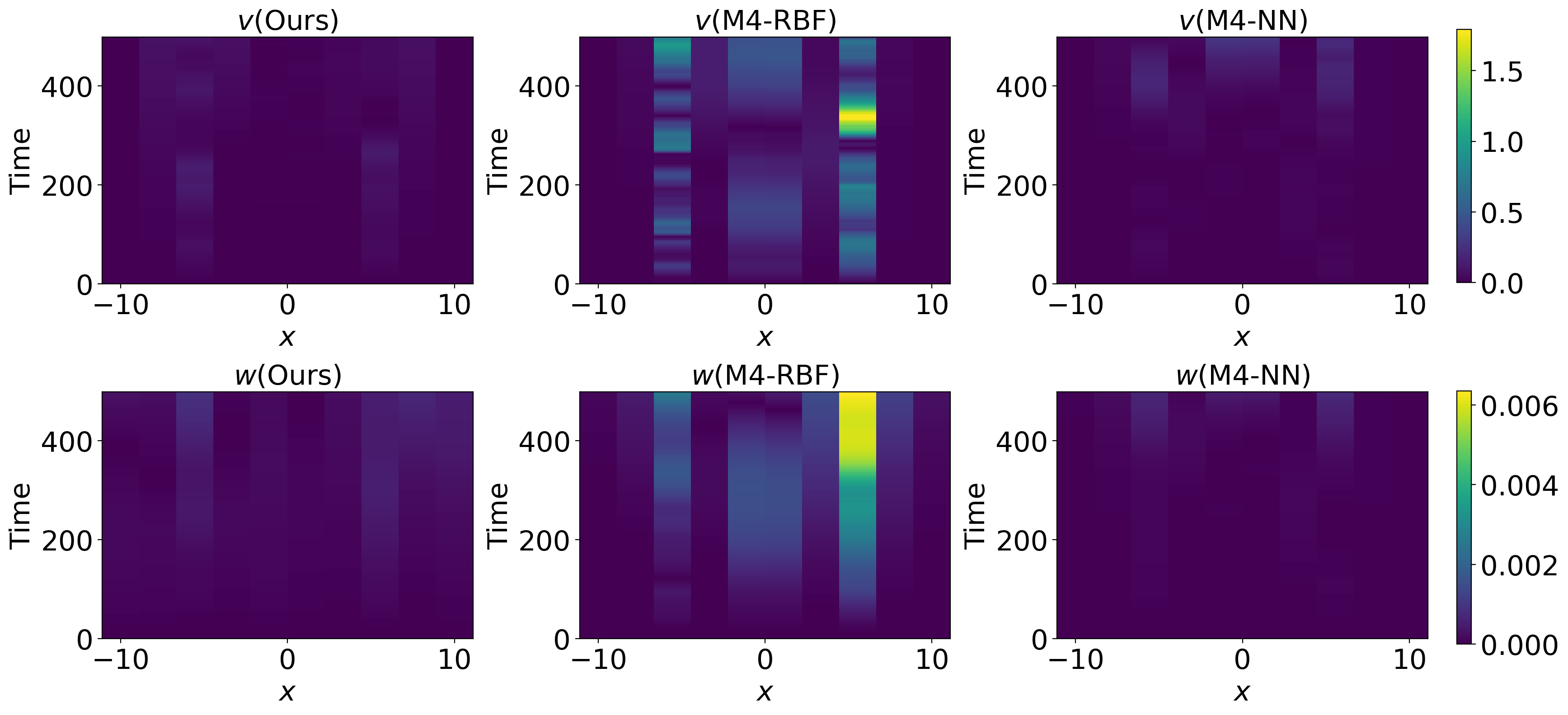}
        \caption{Difference.}
        \label{fig:dim_10_high_u_fhn_diff}
    \end{subfigure}
    \caption{(a): Predictions by PK-NN, M4-RBF and M4-NN on FHN system with $N_{\psi}=10$ and dim $\uu = 3$. (b): The absolute value of the difference between the forecasted outcomes and the ground truth.}
    \label{fig:fhn_dim_10_high_u_traj_diff}
\end{figure}

\begin{figure}[H]
    \centering
    \begin{subfigure}{1.0\textwidth}
    \centering
    \includegraphics[width=1.0\textwidth]{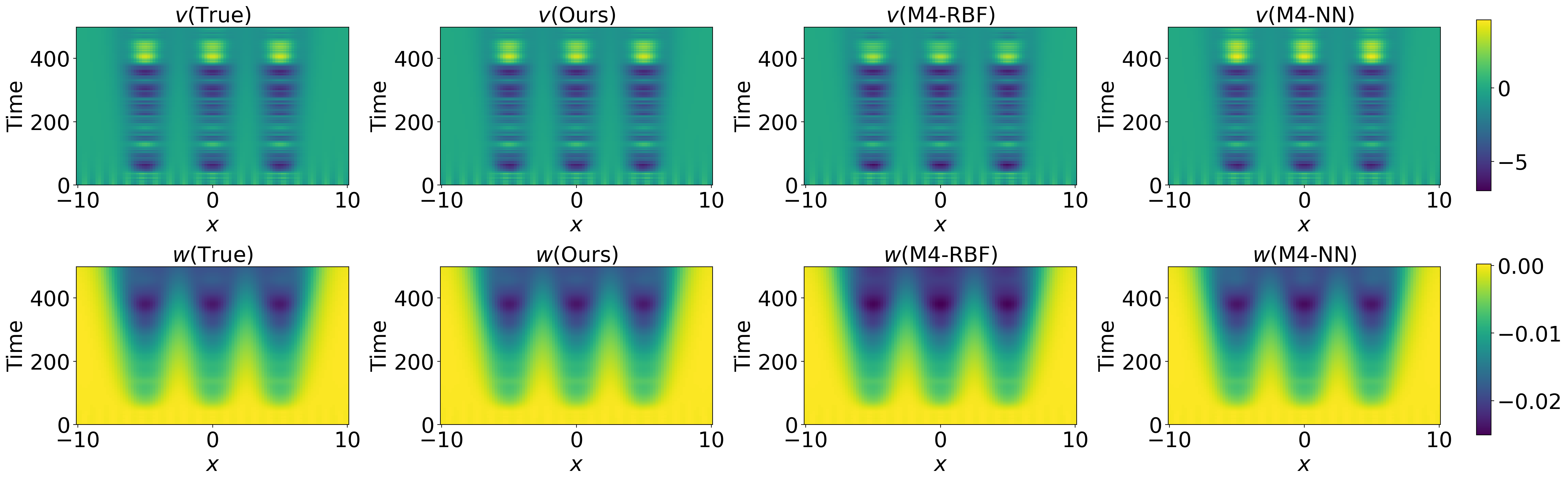}
        \caption{Trajectory.}
        \label{fig:dim_100_fhn_traj}
    \end{subfigure}
    \centering
    \begin{subfigure}{1.0\textwidth}
    \centering
    \includegraphics[width=1.0\textwidth]{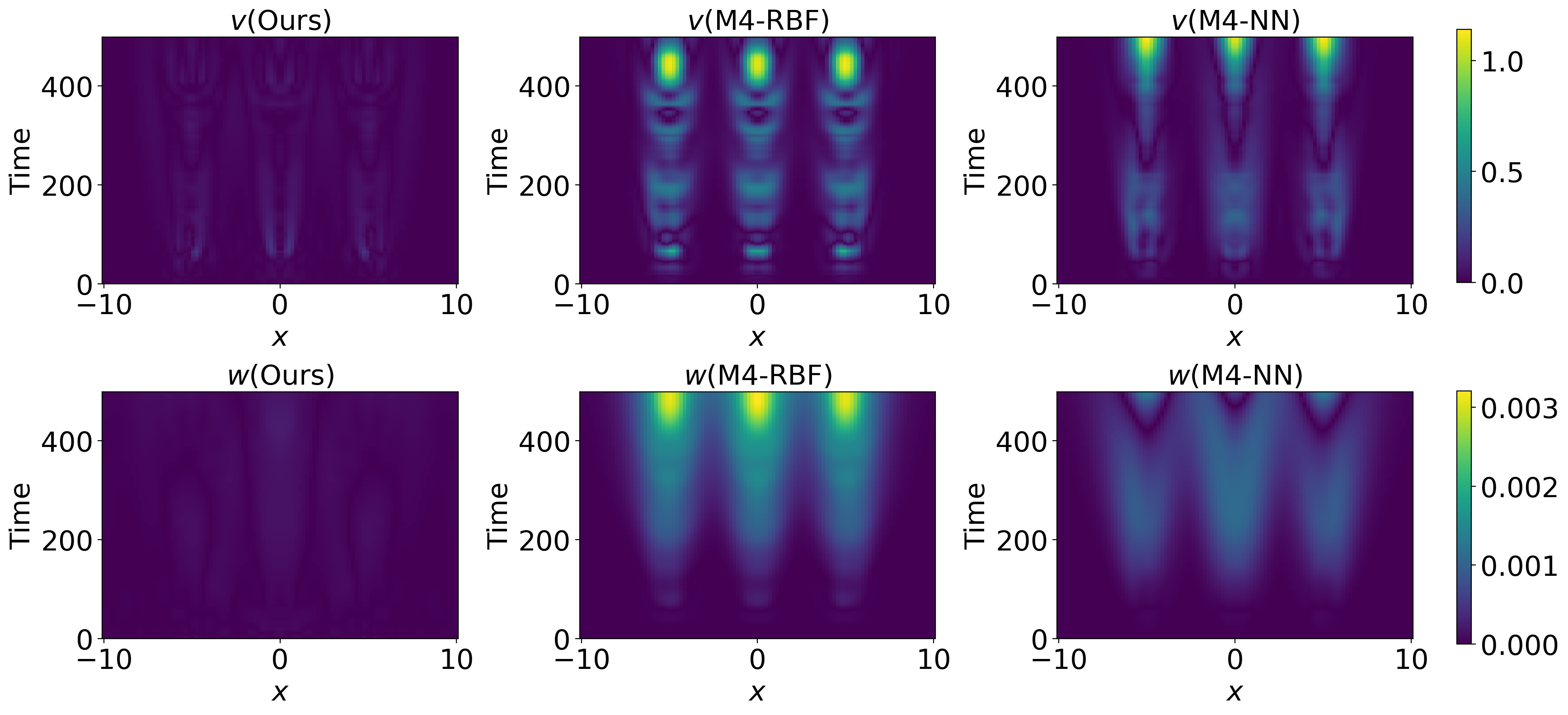}
        \caption{Difference.}
        \label{fig:dim_100_fhn_diff}
    \end{subfigure}
    \caption{(a): Predictions by PK-NN, M4-RBF and M4-NN on FHN system with $N_{\psi}=100$ and dim $\uu = 1$. (b): The absolute value of the difference between the forecasted outcomes and the ground truth.}
    \label{fig:fhn_dim_100_traj_diff}
\end{figure}

\section{Latent space evolution of the Korteweg-De Vries system}
We show the evolution on the latent Koopman space of \pk \ for the Korteweg-De Vries (KdV) system in~\cref{fig:kdv_latent_each_componet}. The solution of KdV equation is generated by a sequence of control $u_i$'s. Controls are applied on the trained \pk, the solution is evolved by it as shown in~\cref{fig:kdv_latent_info} and the predicted mass and momentum are shown at the bottom. The evolution of entries in $K(\uu)$ as $\uu$ changes is shown in~\cref{fig:kdv_latent_each_K}. 
The evolution of ``mass'' mainly depends on the constant term, ``mass'' itself and $NN2$, multiplied by the second row's first, second and fifth element in $K(\uu)$. The multiplications of the entry in $K(\uu)$ and its corresponding element in the dictionary are over $10^{-4}$ for these three components, larger than for others. The evolution of ``momentum'' mainly depends on the components of the latent space except ``mass'' since only the multiplication of the second element of the third row of $K(\uu)$ and ``mass'' is around $10^{-6}$.
We also find that the evolution of $NN1$ and $NN2$ have a linear relationship in~\cref{fig:kdv_latent_NN1_NN2}, which inspires us to figure out how to reduce the redundant ones in the dictionary.

% We find that the evolution of ``mass'' mainly depends on itself and the non-linear term only about $\uu$, which is evaluated from the second row in~\cref{fig:kdv_latent_each_K}. The evolution of ``momentum'' mainly depends on itself and the first trainable component of the latent space (NN1), observed from the third row in~\cref{fig:kdv_latent_each_K}.
% Moreover, we find that the evolution of $NN1$ and $NN2$ have a linear relationship in~\cref{fig:kdv_latent_each_componet}, which insights us to figure out how to reduce the redundant ones in the dictionary.
% We have further investigated how varying dictionary sizes impact model performance and observed that reducing the dictionary size tends to diminish effectiveness. Future work will delve into optimizing dictionary size selection. For now, it is evident that enlarging the dictionary size significantly enhances performance.

\label{sec:kdv_latent}
\begin{figure}[H]
	\centering
	\includegraphics[width=1.0\textwidth]{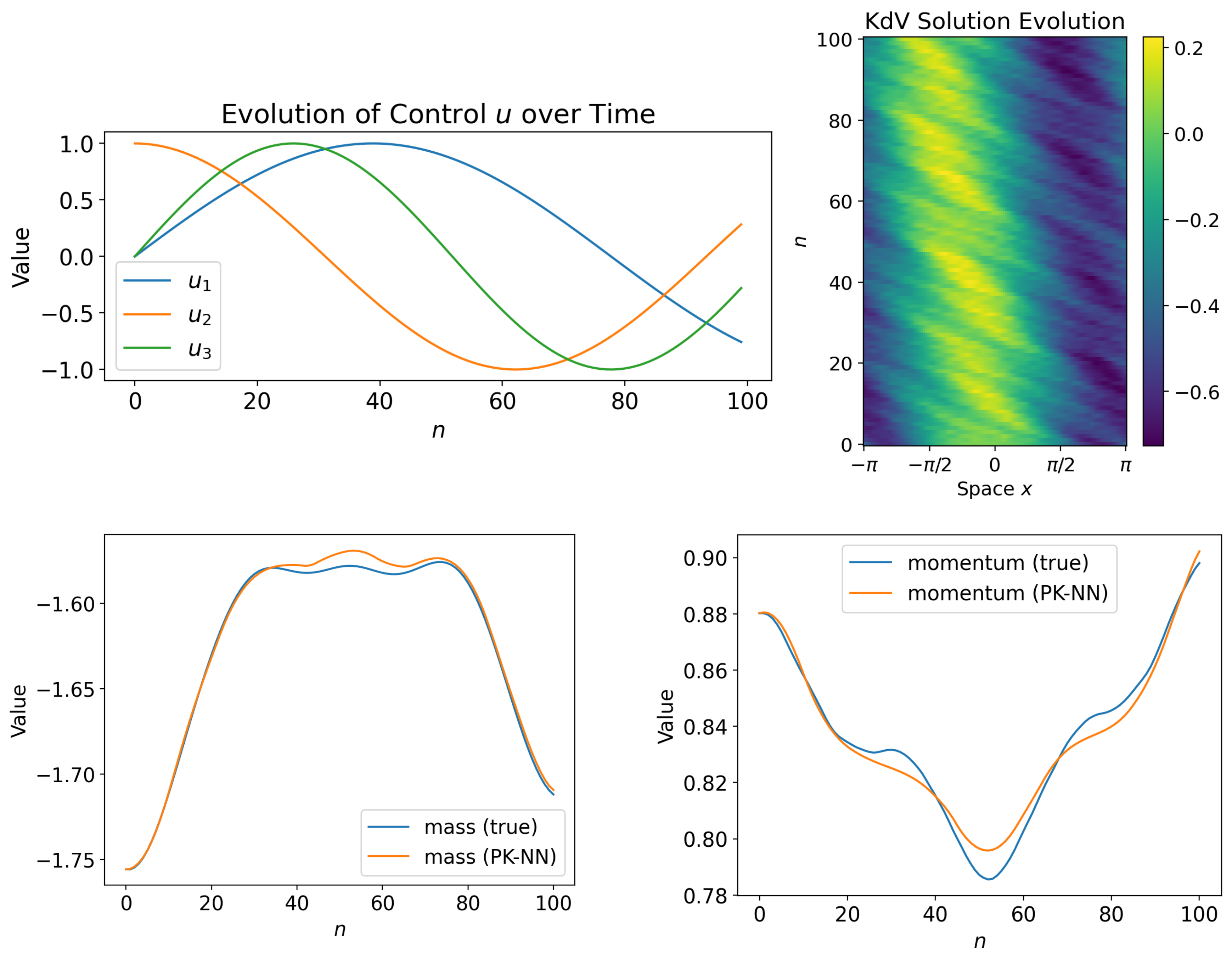}
	\caption{The solution of KdV equation is generated by a sequence of control $u_i$'s. Controls are applied on the trained PK-NN model and the latent space is evolved by the Koopman operator. The predicted mass and momentum are shown at the bottom.}
	\label{fig:kdv_latent_info}
\end{figure}

\begin{figure}[H]
	\centering
	\includegraphics[width=1.0\textwidth]{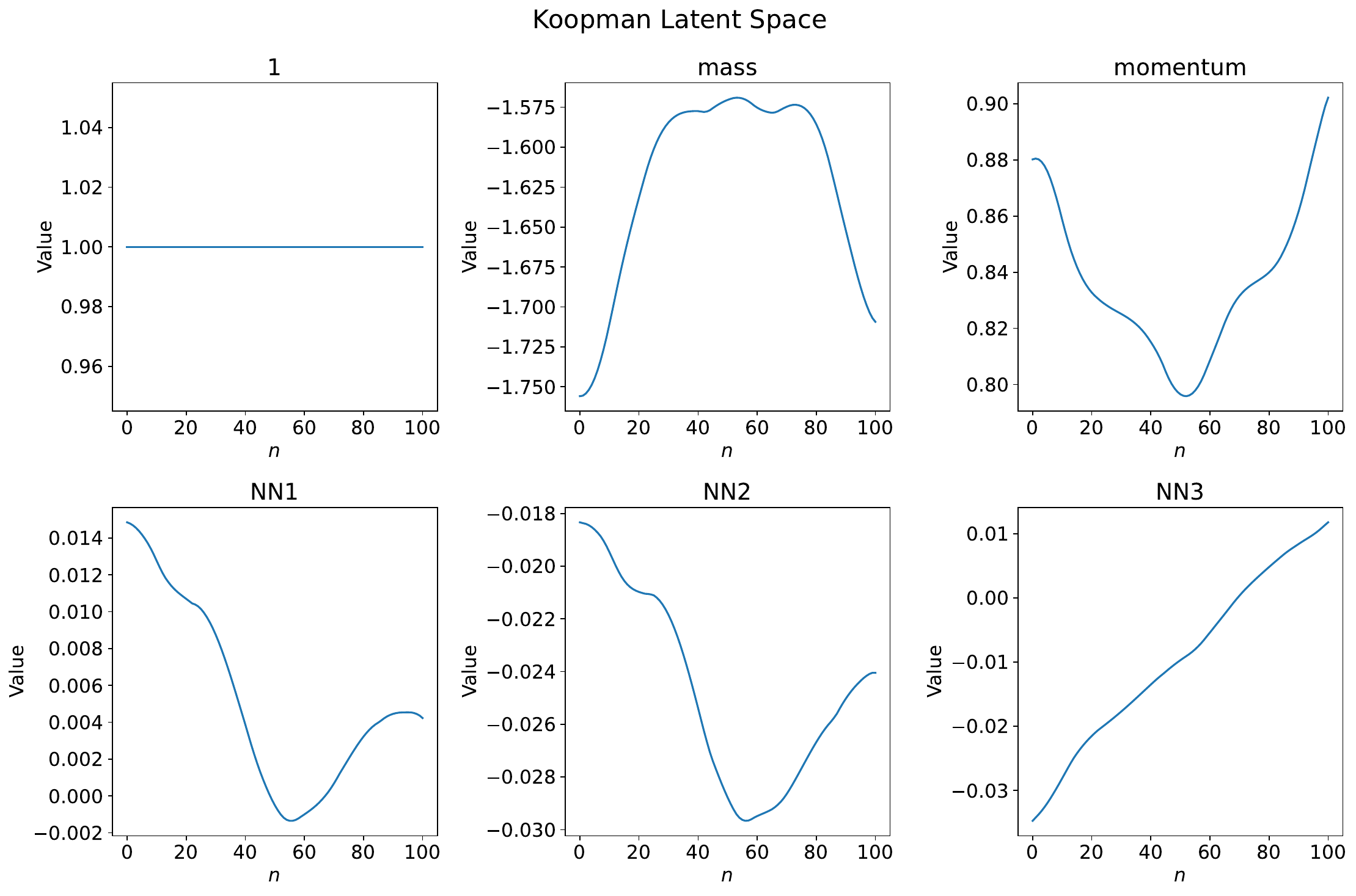}
	\caption{Evolution on the latent space for the Korteweg-De Vries System. The ``mass'' and ``momentum'' are target observables. $NN_{i}, i = 1,2,3$ are trainable components in the dictionary.}
	\label{fig:kdv_latent_each_componet}
\end{figure}

\begin{figure}[H]
	\centering
	\includegraphics[width=1.0\textwidth]{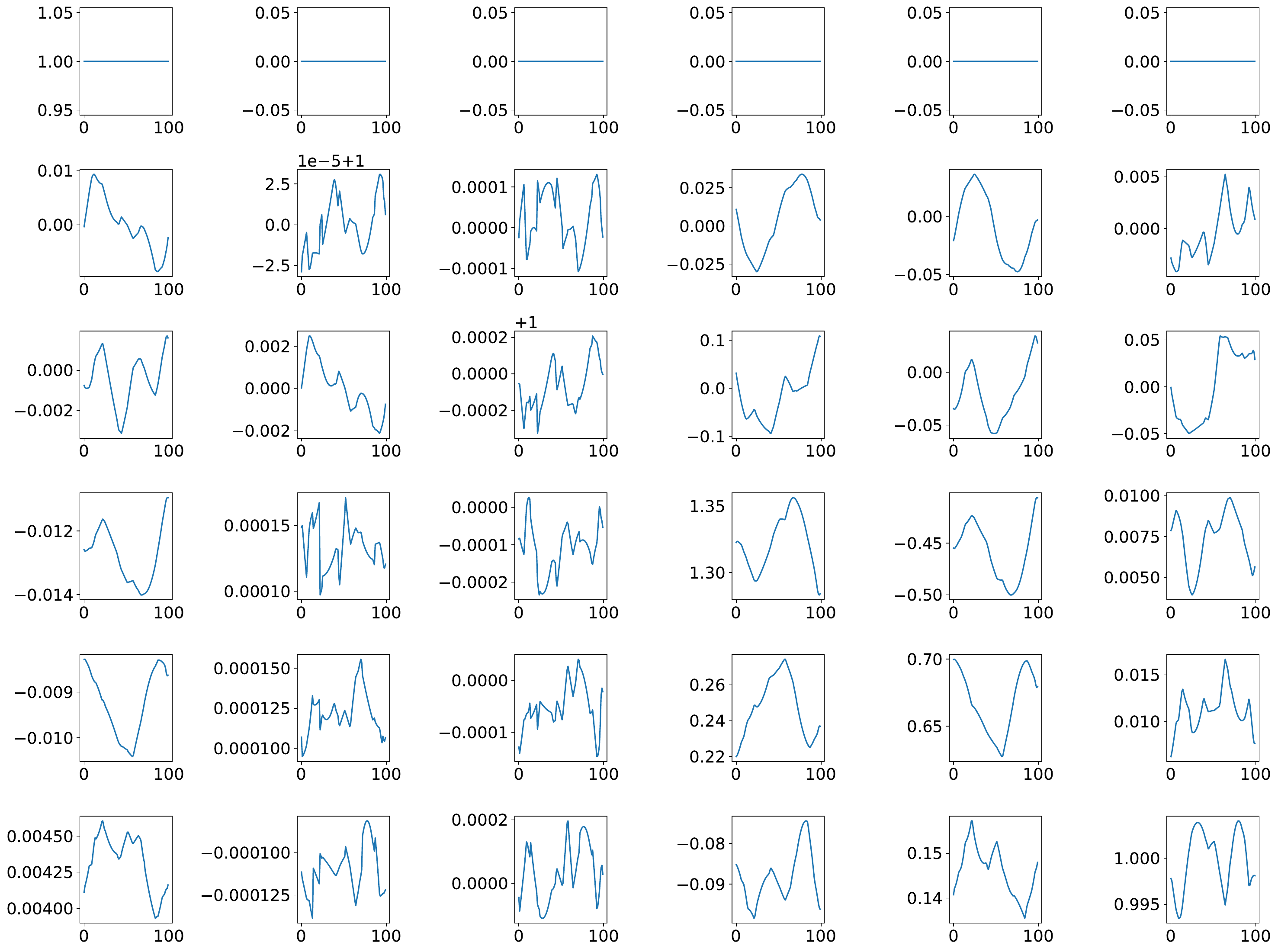}
	\caption{The evolution of entries in $K(\uu)$ as $\uu$ changes. The first row is set as $(1,0,\dots,0)^{T}$ and others are trainable. In this matrix, the diagonal elements corresponding to the second and third rows are optimized to approximate values close to one, a configuration that aligns well with predictions related to mass and momentum. The forecast for the subsequent step is primarily influenced by the current state, along with some other non-linear terms about $\uu$ and $\x$.}
	\label{fig:kdv_latent_each_K}
\end{figure}

\begin{figure}[H]
    \centering
    \includegraphics[width=0.8\textwidth]{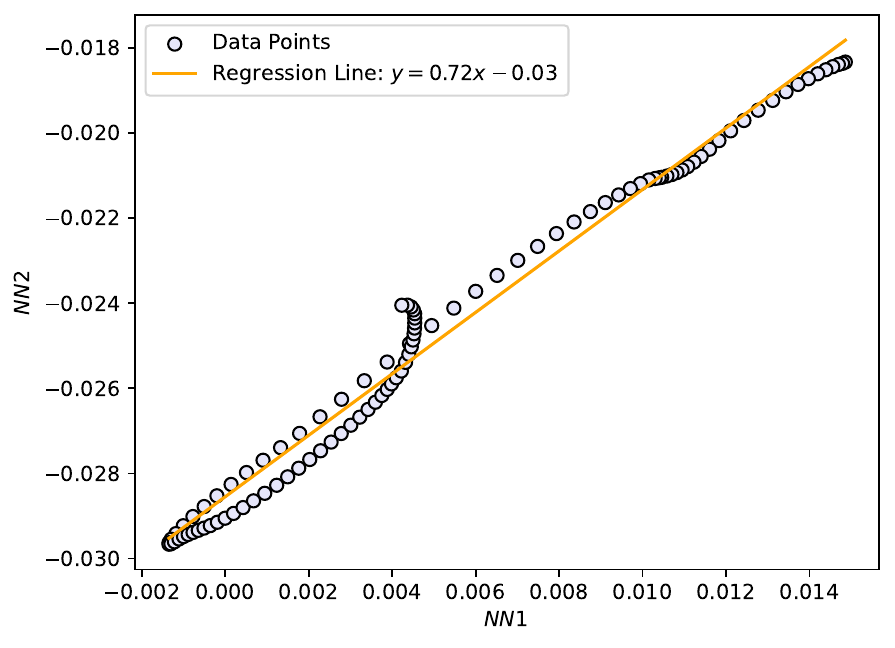}
    \caption{The relationship between $NN1$ and $NN2$.}
    \label{fig:kdv_latent_NN1_NN2}
\end{figure}

\section{Prediction results for mass and momentum of Korteweg-De Vries system}
\label{sec:kdv_mass_momentum_pred}
We compare the predictions of mass and momentum for KdV system by \pk, linear (M2) and bi-linear (M3) in~\cref{fig:kdv_mass_momentum_pred}.
Using~\cref{alg:prediction} and the metric defined in~\cref{eq:rel_error}, 
we compute the relative error based on five trajectories with different initial conditions and plot the corresponding error bars. Observing over 10 time steps, which matches the horizon $\tau$ in MPC, we find that \pk \ outperforms the other two models in both mass and momentum predictions.
This demonstrates that \pk \ can predict the observables $m_{i+1}$ in~\cref{opt_mpc} more accurately, laying a strong foundation for addressing the control problems.

\begin{figure}[H]
    \centering
    \begin{subfigure}{0.48\textwidth}
    \centering
    \includegraphics[width=1.0\textwidth]{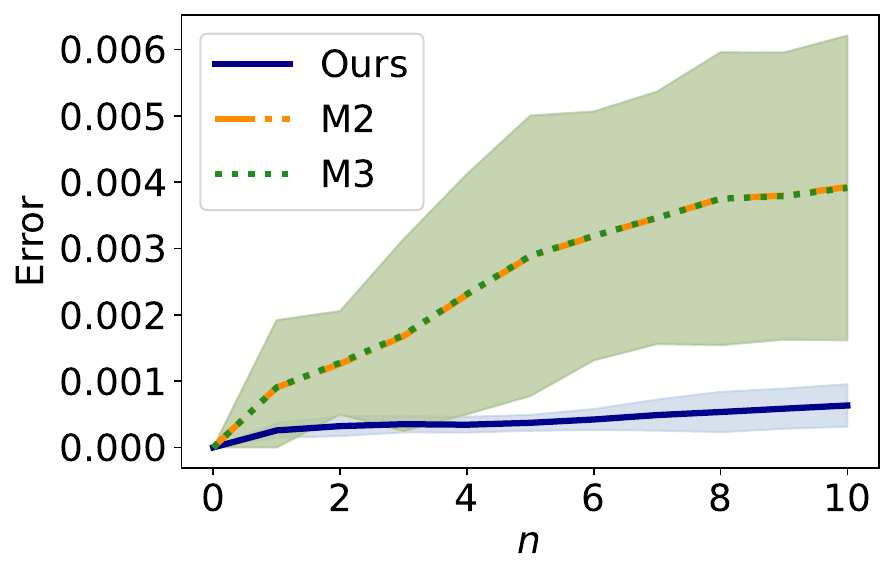}
    \caption{Mass.}
    \label{fig:kdv_mass_pred}
    \end{subfigure}
    \begin{subfigure}{0.48\textwidth}
        \centering
        \includegraphics[width=1.0\textwidth]{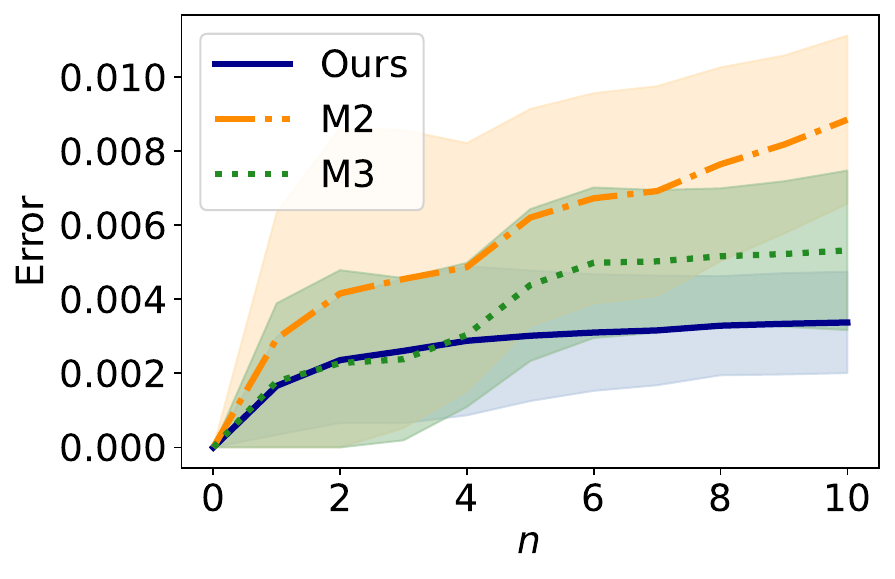}
    \caption{Momentum.}
    \label{fig:kdv_momentum_pred}
    \end{subfigure}
    \caption{The prediction results for mass and momentum of KdV system by \pk, M2 and M3.}
    \label{fig:kdv_mass_momentum_pred}
\end{figure}

\section{\pk \ hyperparameter selection via cross-validation} 
\label{sec:hyperparam_selection}
In this section, we provide a detailed explanation of the hyperparameter selection process for \pk.
We employ cross-validation with grid search to assess the model's performance across various hyperparameters and train models with the same training strategies, including rules for learning rate decay and saving checkpoint criteria. We have a comprehensive explanation of how we choose the hyperparameters. The evaluation focuses on different aspects for each system:
\begin{itemize}
    \item 
    To investigate the impact of dictionary size on model performance, we employ the parametric Duffing system as a case study. It is generally presumed that a sufficiently large dictionary is crucial for accurately capturing the dynamics of the system. Nonetheless, training becomes increasingly challenging with larger dictionary sizes. We assessed the model's performance using various dictionary sizes 5, 10, 15, 20, and 25, and observe improvements in performance with each increment in dictionary size. Consequently, we selected a dictionary size of 25 for our experimental framework The results are shown in~\cref{table:cv_duffing}.
    \item To explore the effect of the number of hidden nodes in $K(\uu)$ on performance, we choose the Van der Pol Mathieu system as an illustrative example, which has strong non-linearity on $\uu$ and $\x$. We evaluate the performance of the model  across various configurations of hidden nodes $[32]$, $[64]$, $[128]$, $[32,32]$, $[64,64]$ and $[128,128]$. We find that two hidden layers do not enhance performance in comparison to a single-layer setup. Consequently, we choose single layer with 128 hidden nodes in our experiments, as this configuration achieves an error magnitude satisfactorily low, around $10^{-9}$. The cross-validatoin results are shown in~\cref{table:cv_vdpm}.
    \item To show the effectiveness in situations involving high-dimensional state spaces, we assess the influence of varying numbers of hidden nodes in the dictionary neural network on the FitzHugh-Nagumo system. Our analysis reveal that the performance of models exhibits sensitivity to changes in the number of hidden nodes, with configurations tested including $[32]$, $[64]$, $[128]$, $[32,32]$, $[64,64]$, and $[128,128]$. Based on our findings, we select a configuration with two hidden layers, each comprising 128 nodes. This configuration leads to a significant improvement in error, around $10^{-7}$, detailed in~\cref{table:cv_fhn}.
\end{itemize}
By evaluating the effect of hyperparameters on neural networks, it often emerges that opting for one or two hidden layers, each with 128 nodes, within the dictionary neural network presents an optimal strategy. Similarly, for the $K(\uu)$ architecture, selecting one or two hidden layers with a width of 128 tends to be a suitable choice. For the activation functions in $K(\uu)$ hidden layers, ``tanh'' works more effectively than ``ReLU'' if the depth is one or two. Regarding the size of the dictionary, our recommendation leans towards choosing a larger size. However, it's important to consider that a larger dictionary results in a more complex $K(\uu)$ matrix, which leads to higher computational demands and greater challenges in training. Therefore, the decision on hyperparameters should be balanced with the performance and the computational capabilities available.

\begin{table}[H]
    \centering
    \footnotesize
    \begin{tabular}{c|ccc}
        \toprule
        \diagbox[dir=NW]{$N_{\psi}$}{Task} & $\nu_1 = 1000, \nu_2 = 10$ & $\nu_1 = 500, \nu_2 = 20$ & $\nu_1 = 100, \nu_2 = 100$ \\
        \midrule
        5 & $1.25 \times 10^{-4}$ & $3.83 \times 10^{-5}$ &$4.74 \times 10^{-4}$  \\
        10 &  $6.46 \times 10^{-5}$ & $1.09 \times 10^{-5}$ &$3.48 \times 10^{-4}$ \\
        15& $4.48 \times 10^{-5}$ & $8.11 \times 10^{-6}$ &$1.29 \times 10^{-4}$\\
        20& $2.10 \times 10^{-5}$ & $4.35 \times 10^{-6}$ &$1.12 \times 10^{-4}$\\
        25& $1.42 \times 10^{-5}$ & $4.68 \times 10^{-6}$ &$1.01 \times 10^{-4}$\\
        \bottomrule
    \end{tabular}
    \caption{Cross-validation results for the case of parametric duffing system.}
    \label{table:cv_duffing}
\end{table}

\begin{table}[H]
    \centering
    \footnotesize
    \begin{tabular}{c|ccccc}
        \toprule
        \diagbox[dir=NW]{$K(\uu)$}{Task} & $\mu=0$ & $\mu=1$ & $\mu=2$ & $\mu=3$ & $\mu=4$ \\
        \midrule
        $[32]$ 
        & $1.08 \times 10^{-8}$
        & $1.34 \times 10^{-8}$
        & $1.27 \times 10^{-8}$
        & $1.76 \times 10^{-8}$
        & $2.26 \times 10^{-8}$  \\
        $[64]$ 
        & $7.12 \times 10^{-9}$
        & $8.12 \times 10^{-9}$
        & $9.28 \times 10^{-9}$
        & $1.09 \times 10^{-8}$
        & $1.98 \times 10^{-8}$  \\
        $[128]$
        & $6.23 \times 10^{-9}$
        & $7.02 \times 10^{-9}$
        & $8.31 \times 10^{-9}$
        & $1.47 \times 10^{-8}$
        & $1.15 \times 10^{-8}$ \\
        $[32,32]$
        & $6.05 \times 10^{-9}$
        & $7.91 \times 10^{-9}$ 
        & $8.43 \times 10^{-9}$
        & $1.49 \times 10^{-8}$
        & $2.71 \times 10^{-8}$ \\
        $[64,64]$ 
        & $6.75 \times 10^{-9}$ 
        & $6.48 \times 10^{-9}$
        & $7.07 \times 10^{-9}$
        & $1.36 \times 10^{-8}$
        & $2.26 \times 10^{-8}$ \\
        $[128,128]$ 
        & $8.28 \times 10^{-9}$
        & $8.80 \times 10^{-9}$
        & $8.29 \times 10^{-9}$
        & $1.36 \times 10^{-8}$ 
        & $2.01 \times 10^{-8}$\\
        \bottomrule
    \end{tabular}
    \caption{Cross-validation results for the case of Van der Pol Matheiu system.}
    \label{table:cv_vdpm}
\end{table}

\begin{table}[H]
    \centering
    \footnotesize
    \begin{tabular}{c|cccccc}
        \toprule
        $\vpsi(\x)$ & $[32]$ & $[64]$ & $[128]$ & $[32,32]$ & $[64,64]$ & $[128,128]$ \\
        \midrule
        Loss
         & $1.82 \times 10^{-6}$ 
         & $6.66 \times 10^{-7}$ 
         & $2.47 \times 10^{-7}$  
         & $7.77 \times 10^{-7}$ 
         & $3.41 \times 10^{-7}$  
         & $1.45 \times 10^{-7}$  \\
        \bottomrule
    \end{tabular}
    \caption{Cross-validation results for the case of FitzHugh-Nagumo system $(N_x  =100, \text{dim} \ \uu = 1)$.}
    \label{table:cv_fhn}
\end{table}
        
\section{Comparison about the computational cost of optimal control problems}
\label{sec:kdv_computational_cost}
We examine the computational costs associated with solving optimal control problems for the Korteweg-De Vries (KdV) system. We consider the tracking momentum problem, with a focus on a system influenced by non-linear (sin) forcing, setting $\lambda=0.005$ as an example. A key strategy in our optimization process is to initialize the parameters to be optimized using the outcomes from the previous step. This iterative initialization significantly decreases the optimization time, particularly when consecutive control values are closely related. 
% The optimization process utilizes the L-BFGS algorithm \cite{liu1989limited} on Python and 3.20GHz Intel Xeon W-3245 with 125 GB RAM.
We show both the optimization time and the iteration count for each step, as detailed in~\cref{fig:kdv_computational_cost}. Our findings indicate that the computational demands of the PK-NN model are comparably higher than those of other models. 
The iteration times for models M2 (linear) and M3 (bilinear) are around 1 second, whereas for the PK-NN model, they vary between 1 and 8 seconds. The optimization averages 5.82 seconds for PK-NN, 0.83 seconds for M2, and 1.69 seconds for M3. Consequently, with the same optimizer in use, the PK-NN model requires 3 to 7 times more time to complete each iteration compared to the other two methods.

\begin{figure}[htb!]
    \centering
    \begin{subfigure}{0.9\textwidth}
    \centering
    \includegraphics[width=0.9\textwidth]{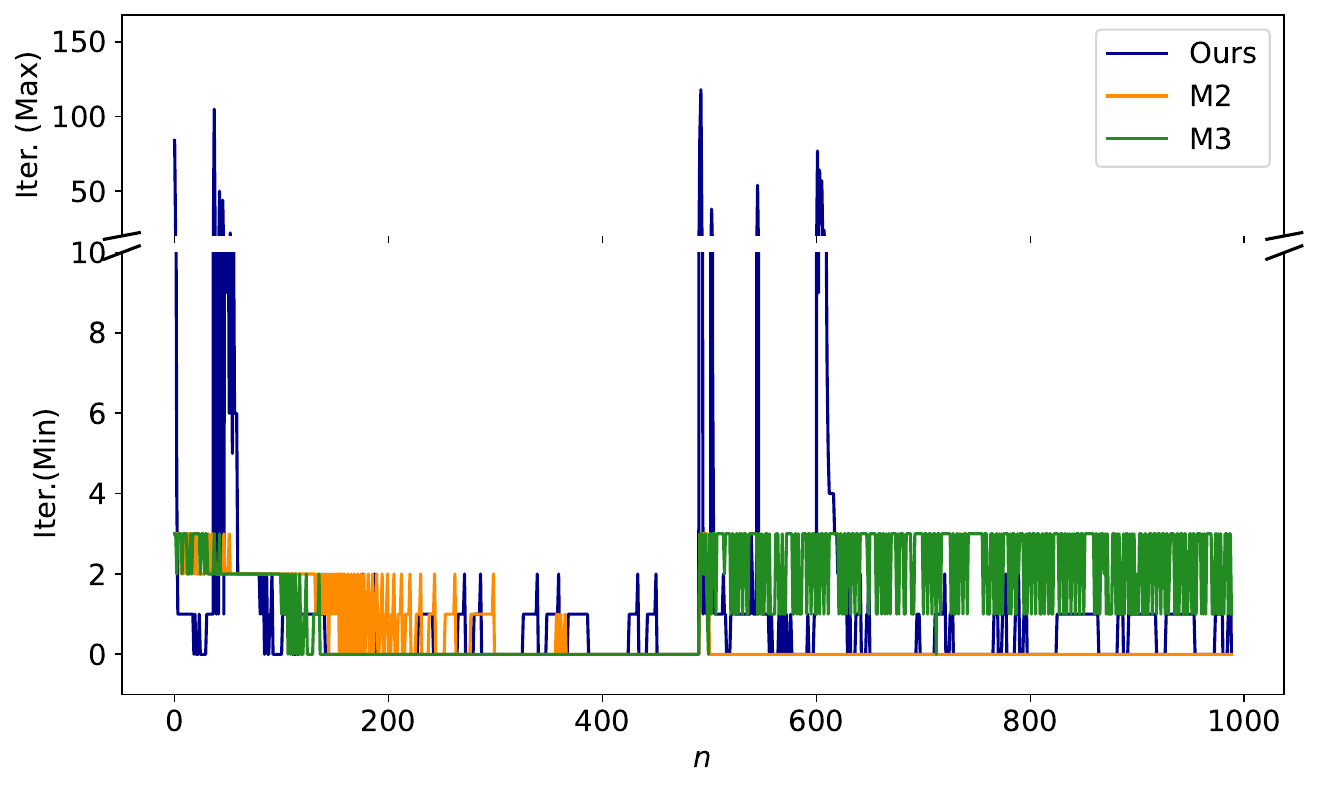}
        \caption{The number of iterations for each optimization.}
        \label{fig:opt_iter}
    \end{subfigure}
    \centering
    \begin{subfigure}{0.9\textwidth}
    \centering
    \includegraphics[width=0.9\textwidth]{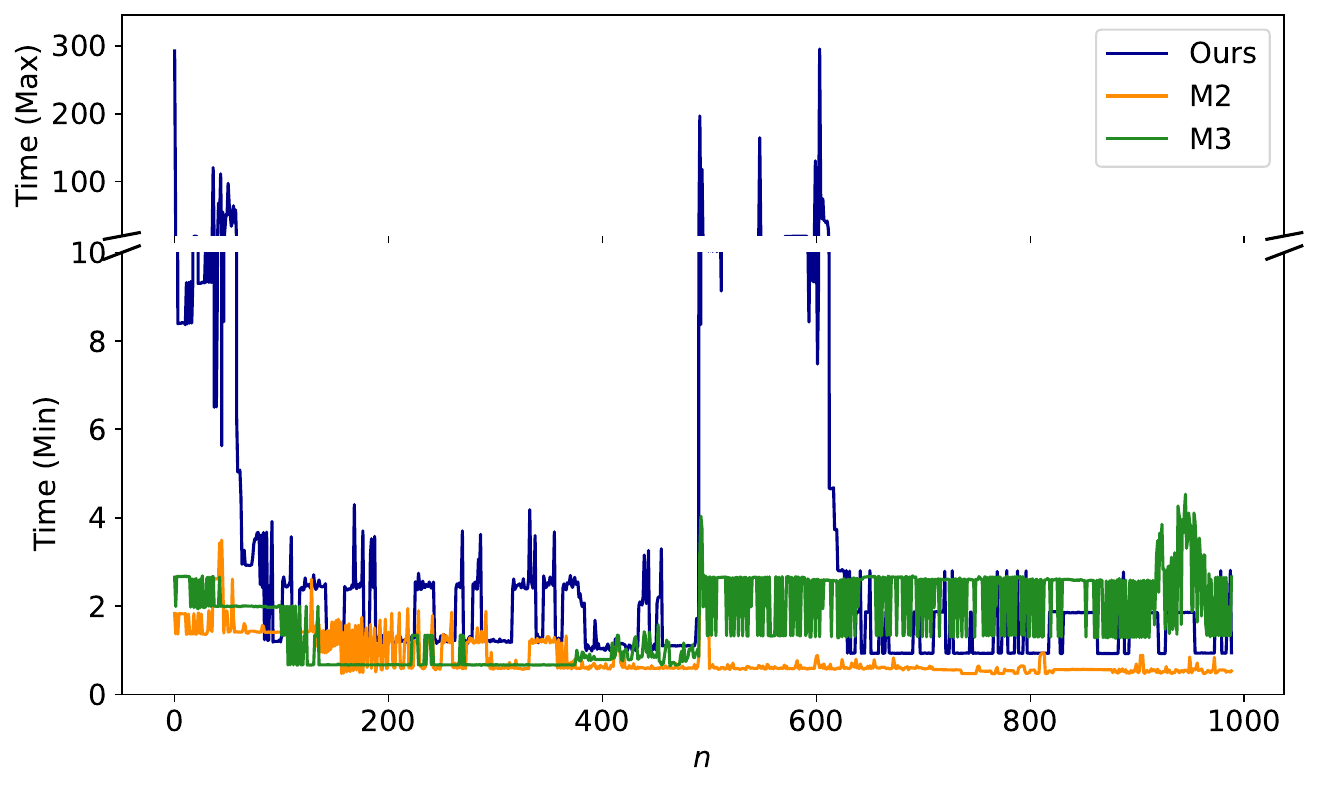}
        \caption{Cost of time for each optimization.}
        \label{fig:opt_t_list}
    \end{subfigure}
    \caption{Computational cost of the optimization in KdV tracking problem by \pk, M2 and M3.}
    \label{fig:kdv_computational_cost}
\end{figure}

\end{document}